\documentclass[12pt]{article}
\usepackage[lmargin=.8in,rmargin=.8in,bmargin=1.2in,tmargin=1.1in]{geometry}
\usepackage[margin=.4in]{caption}

\usepackage{titlesec}
\titleformat{\section}{\large\bfseries}{\thesection}{1em}{}
\titleformat{\subsection}{\large}{\thesubsection}{1em}{}

\usepackage{amsmath}
\usepackage{amssymb}
\usepackage{mathrsfs}
\usepackage{verbatim}
\usepackage{dsfont}
\usepackage{amsthm}
\usepackage{tikz}
\usetikzlibrary{arrows}
\usepackage{authblk}
\usepackage{times}

\newtheorem{theorem}{Theorem}
\newtheorem{lemma}[theorem]{Lemma}
\newtheorem{corollary}[theorem]{Corollary}
\newtheorem{proposition}[theorem]{Proposition}

\theoremstyle{definition}

\newtheorem{definition}[theorem]{Definition}
\newtheorem{remark}[theorem]{Remark}

\newcommand{\Bbf}{\mathbf{B}}
\newcommand{\Pbf}{\mathbf{P}}
\newcommand{\Qbf}{\mathbf{Q}}

\newcommand{\ol}[1]{\overline{#1}}

\newcommand{\Ecal}{\mathcal{E}}
\newcommand{\Fcal}{\mathcal{F}}
\newcommand{\Mcal}{\mathcal{M}}

\newcommand{\Xcal}{\mathcal{X}}

\newcommand{\E}{\mathbb{E}}
\newcommand{\R}{\mathbb{R}}

\DeclareMathOperator{\conv}{conv}
\DeclareMathOperator{\cov}{cov}
\DeclareMathOperator{\ver}{vert}

\DeclareMathOperator{\argmax}{argmax}
\DeclareMathOperator{\argmin}{argmin}

\DeclareMathOperator*{\bigtimes}{\textnormal{\Large $\times$}} % Cartesian Product

\definecolor{bl}{RGB}{20,20,200}
\usepackage{hyperref}
\hypersetup{
 colorlinks=true,
 citecolor=bl,
 linkcolor=bl,
 urlcolor=bl}

\begin{document}

\title{\bf On the Fisher Metric of Conditional Probability Polytopes}
\author[1]{\rm Guido Mont\'ufar}
\author[1]{\rm Johannes Rauh}
\author[1,2,3]{\rm Nihat Ay}
\affil[{ }]{$\{$montufar,\,jrauh,\,nay$\}$@mis.mpg.de}
\affil[1]{\small Max Planck Institute for Mathematics in the Sciences, Inselstra\ss e 22, 04103 Leipzig, Germany}
\affil[2]{\small Department of Mathematics and Computer Science, Leipzig University, PF 10 09 20, 04009 Leipzig, Germany}
\affil[3]{\small Santa Fe Institute, 1399 Hyde Park Road, Santa Fe, NM, 87501, USA}
\date{ }

\maketitle
\thispagestyle{empty}
\abstract{ %
We consider three different approaches to define natural Riemannian metrics on polytopes of stochastic matrices. First, we define a natural class of stochastic maps between these polytopes and give a metric characterization of Chentsov type in terms of invariance with respect to these maps. Second, we consider the Fisher metric defined on arbitrary polytopes through their embeddings as exponential families in the probability simplex. We show that these metrics can also be characterized by an invariance principle with respect to morphisms of exponential families. Third, we consider the Fisher metric resulting from embedding the polytope of stochastic matrices in a simplex of joint distributions by specifying a marginal distribution. All three approaches result in slight variations of products of Fisher metrics. This is consistent with the nature of polytopes of stochastic matrices, which are Cartesian products of probability simplices. The first approach yields a scaled product of Fisher metrics; the second, a product of Fisher metrics; and the third, a product of Fisher metrics scaled by the marginal distribution.

\medskip

}
%\PACS{}
%\MSC{}

\section{Introduction}

The Riemannian structure of a function's domain has a crucial impact on the performance of gradient optimization methods, especially in the presence of plateaus and local maxima.
The natural gradient~\cite{Amari} gives the steepest increase direction of functions on a Riemannian space.
For example, artificial neural networks can often be trained by following some function's gradient on a space of probabilities.
In this context, it has been observed that following the natural gradient with respect to the Fisher information metric, instead of the Euclidean metric, can significantly alleviate the plateau problem~\cite{Amari,Kakade01}.
The Fisher information metric, which is also called Shahshahani metric~\cite{shahshahani1979new} in biological contexts, is broadly recognized as the natural metric of probability spaces.
An important argument was given by Chentsov~\cite{cencov},
who showed that the Fisher information metric is the only metric on probability spaces for which certain natural statistical embeddings, called Markov morphisms, are isometries.
More generally, Chentsov's theorem characterizes the Fisher metric and $\alpha$-connections of statistical manifolds uniquely (up to a multiplicative constant) by requiring invariance with respect to Markov morphisms.
Campbell~\cite{Campbell} gave another proof that characterizes invariant metrics on the set of non-normalized positive measures, which restrict to the Fisher metric
 in the case of probability measures (up to a multiplicative constant).
In this paper, we explore ways of defining distinguished Riemannian metrics on spaces of stochastic matrices.

In learning theory, when modeling the policy of a system, it is often preferred to consider stochastic matrices instead of joint probability distributions.
For example, in robotics applications, policies are optimized over a parametric set of stochastic matrices by following the gradient of a reward \scalebox{.95}[1.0]{function~\cite{Sutton00policygradient,Marbach}.}
The set of stochastic matrices can be parametrized in many ways, e.g., in terms of feedforward neural networks, Boltzmann machines~\cite{montufar2014CRBMs} or projections of exponential families~\cite{AyMontufarRauhICCN2011}.
The information geometry of policy models plays an important role in these applications and has been studied by Kakade~\cite{Kakade01}, Peters and co-workers~\cite{Peters20081180,Peters2006, peters2003reinforcement}, and Bagnell and Schneider~\cite{Bagnell:2003:CPS:1630659.1630805}, 
% and Vijayakumar \emph{et al}.~\cite{Vijayakumar03reinforcementlearning}, 
among others.
A stochastic matrix is a tuple of probability distributions, and therefore, the space of stochastic matrices is a Cartesian product of probability simplices.
Accordingly, in applications, usually a product metric is considered, with the usual Fisher metric on each factor.
On the other hand, Lebanon~\cite{Lebanon05} takes an axiomatic approach, following the ideas of Chentsov and Campbell,
and characterizes a class of invariant metrics of positive matrices that restricts to the product of Fisher metrics in the case of stochastic matrices.
We will consider three different approaches discussed in the following.

In the first part, we take another look at Lebanon's approach
for characterizing a distinguished metric on polytopes of stochastic matrices.
However, since the maps considered by Lebanon do not map stochastic matrices to stochastic matrices, we will use different maps.
We show that the product of Fisher metrics can be characterized by an invariance principle with respect to natural maps between stochastic matrices.

In the second part, we consider an approach that allows us to define Riemannian structures on arbitrary polytopes.
Any polytope can be identified with an exponential family by using the coordinates of the polytope vertices as observables.
The inverse of the moment map then defines an embedding of the polytope in a probability simplex.
This embedding can be used to pull back geometric structures from the probability simplex to the polytope, including Riemannian metrics, affine connections, divergences,~\emph{etc}.
This approach has been considered in~\cite{AyMontufarRauhICCN2011} as a way to define low-dimensional families of conditional probability distributions.
More general embeddings can be defined by identifying each exponential family with a point configuration, $\Bbf$, together with a weight function,~$\nu$.
Given $\Bbf$ and $\nu$, the corresponding exponential family defines geometric structures on the set $(\conv\Bbf)^\circ$, which is the relative interior of the convex support of the exponential family.
Moreover, we can define natural morphisms between weighted point configurations as surjective maps between the point sets, which are compatible with the weight functions. As it turns out, the Fisher metric on $(\conv\Bbf)^\circ$ can be characterized by invariance under these maps.

In the third part, we return to stochastic matrices. We study natural embeddings of conditional distributions in probability simplices as joint distributions with a fixed marginal.
These embeddings define a Fisher metric equal to a weighted product of Fisher metrics. This result corresponds to the definitions commonly used in robotics applications.

All three approaches give very similar results.
In all cases, the identified metric is a product metric.
This is a sensible result, since the set of $k\times m$ stochastic matrices is a Cartesian product of probability simplices $\Delta_{m-1}\times\dots\times\Delta_{m-1}=\Delta_{m-1}^{k}$, which suggests using the product metric of the Fisher metrics defined on the factor simplices, $\Delta_{m-1}$.
Indeed, this is the result obtained from our second approach.
The first approach yields that same result with an additional scaling factor of~$1/k$.
Only when stochastic matrices of different sizes are compared, the two approaches differ.
The third approach yields a product of Fisher metrics scaled by the marginal distribution that defines the embedding.

Which metric to use depends on the concrete problem and whether a natural marginal distribution is defined and known.
In Section~\ref{sec:replicator-example}, we do a case study using a reward function that is given as an expectation value
over a joint distribution. In this simple example, the weighted product metric gives the best asymptotic rate of convergence,
under the assumption that the weights are optimally chosen. In Section~\ref{sec:conclusions}, we sum up our findings.

The contents of the paper is organized as follows.
Section~\ref{sec:prelim} contains basic definitions around the Fisher metric and concepts of differential geometry.
In Section~\ref{sec:Campbell-Lebanon}, we discuss the theorems of Chentsov, Campbell and Lebanon, which characterize natural geometric structures on the probability simplex, on the set of positive measures and on the cone of positive matrices, respectively.
In Section~\ref{section:invariance}, we study metrics on polytopes of stochastic matrices, which are invariant under natural embeddings.
In Section~\ref{sec:polytopes-Fisher}, we define a Riemannian structure for polytopes, which generalizes the Fisher information metric of probability simplices and conditional models in a natural way.
In Section~\ref{sec:weighted-metrics}, we study a class of weighted product metrics.
In Section~\ref{sec:replicator-example}, we study the gradient flow with respect to an expectation value.
Section~\ref{sec:conclusions} contains concluding remarks.
In Appendix~\ref{section:positive}, we investigate restrictions on the parameters of the metrics characterized in Sections~\ref{sec:Campbell-Lebanon} and~\ref{section:invariance} that make them positive definite. Appendix~\ref{sec:proofs-invariance} contains the proofs of the results from Section~\ref{section:invariance}.

\section{Preliminaries}
\label{sec:prelim}

We will consider the simplex of probability distributions on $[m]:=\{1,\ldots, m\}$, $m\geq 2$, which is given by
$\Delta_{m-1}:=\{(p_i)_i\in\R^{m} \colon p_i \geq 0, \sum_i p_i=1 \}$.
The relative interior of $\Delta_{m-1}$ consists of all strictly positive probability distributions on $[m]$,
and will be denoted $\Delta_{m-1}^\circ$. This is a subset of $\R_{+}^{m}$, the cone of strictly positive vectors.
The set of $k \times m$ row-stochastic matrices is given by
$\Delta_{m-1}^k:=\{(K_{ij})_{ij} \in\R^{k\times m} \colon (K_{ij})_j\in\Delta_{m-1}\text{ for all } i\in[k] \}$ and is equal to the Cartesian product $\bigtimes_{i\in[k]}\Delta_{m-1}$. The relative interior $( \Delta_{m-1}^k)^\circ$ is a subset of $\R_+^{k\times m}$, the cone of strictly positive matrices.

Given two random variables $X$ and $Y$ taking values in the finite sets $[k]$ and $[m]$, respectively, the conditional probability distribution of $Y$ given $X$ is the stochastic matrix $K=(P(y|x))_{x\in [k],y\in [m]}$ with rows $(P(y|x))_{y\in[m]}\in \Delta_{m-1}$ for all $x\in[k]$.
Therefore, the polytope of stochastic matrices $\Delta_{m-1}^k$ is called a conditional polytope.

The tangent space of $\R^{n}_{+}$ at a point $p\in\R^{n}_{+}$, denoted by $T_{p}\R^{n}_{+}$, is the real vector space
spanned by the vectors $\partial_{1},\ldots,\partial_{n}$ of partial derivatives with respect to the $n$ components.
The tangent space of $\Delta_{n-1}^\circ$ at a point $p\in\Delta_{n-1}^\circ\subset\R^{n}_{+}$
is the subspace $T_{p}\Delta_{n-1}^\circ\subset T_{p}\R_+^n$ consisting of the vectors:
\begin{equation}
 u=\sum_{i}u_{i}\partial_{i} \in T_{p}\R^{n}_{+}\quad \text{with}\quad \sum_{i}u_{i}=0.
\end{equation}

The Fisher metric on the positive probability simplex $\Delta_{n-1}^\circ$ is the Riemannian metric given by:
\begin{equation}
 g^{(n)}_{p}(u,v) = \sum_{i=1}^{n}\frac{u_{i}v_{i}}{p_i},\quad\text{for all $u,v\in T_p\Delta_{n-1}^\circ$}. \label{eq:Fishermetric}
\end{equation}
The same formula~\eqref{eq:Fishermetric} also defines a Riemannian metric on $\R^{n}_{+}$, which we will denote by the same symbol.
This, however, is not the only way in which the Fisher metric can be extended from $\Delta_{n-1}^\circ$ to~$\R^{n}_{+}$.
We will discuss other extensions in the next section (see Campbell's theorem, Theorem~\ref{thm:Campbell}).

Consider a smoothly parametrized family of probability distributions $\Mcal=\{(p(x;\theta))_{x\in [n]}\colon \theta\in \Omega\}\subseteq \Delta_{n-1}^\circ$, where $\Omega\subseteq\R^d$ is open.
Then, $g^{(n)}$ induces a Riemannian metric on~$\Mcal$.
Denote by $\partial_{\theta_{i}}=\frac{\partial}{\partial\theta_i}$ the tangent vector corresponding to the partial derivative with respect to $\theta_i$, for all $i\in[d]$.
Then, the Fisher matrix has coordinates:
\begin{equation}
g^{\Mcal}_\theta(\partial_{\theta_{i}}, \partial_{\theta_{j}}) =
\sum_{x\in [n]} p(x;\theta) \frac{ \partial \log p(x;\theta)}{\partial\theta_i} \frac{ \partial \log p(x;\theta)}{\partial\theta_j}, \quad \text{for all $i,j\in[d]$}, \quad\text{for all $\theta\in\Omega$}.
\end{equation}
Here, it is not necessary to assume that the parameters $\theta_{i}$ are independent.
In particular, the dimension of $\Mcal$ may be smaller than $d$, in which case the matrix is not positive definite.
If the map $\Omega\to\Mcal, \theta\mapsto p( \cdot;\theta)$ is an embedding (\emph{i.e.},~a smooth injective map that is a diffeomorphism onto its image), then $g^{\Mcal}_{\theta}$ defines a Riemannian metric on $\Omega$, which corresponds to the pull-back of~$g^{(n)}$.

Consider an embedding $f\colon \mathcal{E}\to\mathcal{E}'$.
The pull-back of a metric $g'$ on $\Ecal'$ through $f$ is defined as:
\begin{equation}
(f^\ast g')_p (u,v) : = g'_{f(p)} (f_\ast u, f_\ast v), \quad\text{for all $u,v\in T_p\Ecal$},
\end{equation}
where $f_\ast$ denotes the push-forward of $T_p\Ecal$ through $f$, which in coordinates is given by:
\begin{equation}
f_\ast\colon \quad T_p \Ecal \to T_{f(p)}\Ecal'; \quad \sum_{i } u_i \partial_{\theta_{i}} \mapsto \sum_{j } \sum_{i } u_i \, \frac{\partial f_j(p)}{\partial \theta_{i}} \, \partial_{\theta_{j}'},
\end{equation}
where $\{\partial_{\theta_{i}}\}_i$ spans $T_q\Ecal$ and $\{\partial_{\theta_{j}'} \}_j$ spans $T_{f(p)}\Ecal'$.

An embedding $f\colon \mathcal{E}\to\mathcal{E}'$ of two Riemannian manifolds $(\mathcal{E},g)$ and $(\mathcal{E}',g')$ is an isometry iff:
\begin{equation}
g_p(u,v) = (f^\ast g')_p(u,v),\quad\text{for all $p\in\mathcal{E}$ and $u,v\in T_p\mathcal{E}$}.
\end{equation}
In this case, we say that the metric $g$ is invariant with respect to $f$ (and $g'$).

\section{The Results of Campbell and Lebanon}
\label{sec:Campbell-Lebanon}

One of the theoretical motivations for using the Fisher metric is provided by Chentsov's characterization~\cite{cencov},
which states that the Fisher metric is uniquely specified, up to a multiplicative constant, by an invariance principle under a class of stochastic maps, called Markov morphisms.
Later, Campbell~\cite{Campbell} considered the characterization problem on the space $\R_{+}^{n}$ instead of $\Delta_{n-1}^\circ$.
This simplifies the computations, since $\R_{+}^{n}$ has a more symmetric parametrization.

\begin{definition}
\label{def:campbellmorphisms}
Let $2\leq m \leq n$.
A (row) stochastic partition matrix (or just row-partition matrix) is a matrix $Q\in\R^{m\times n}$ of non-negative entries, which satisfies
$\sum_{j\in A_{i'}} Q_{ij} = \delta_{ii'}$ for an $m$ block partition $\{A_1,\ldots, A_m \}$ of $[n]$.
The linear map defined by:
\begin{equation}
\R_+^m\to \R_+^n; \quad p\mapsto p \cdot Q
\end{equation}
is called a congruent embedding by a Markov mapping of $\R_+^m$ to $\R_+^n$ or just a Markov map, for short.
\end{definition}
An example of a $3\times 5$ row-partition matrix is:
\begin{equation}
 \label{eq:stoch-part-mat}
Q = \begin{pmatrix}
1/2 & 0 & 1/2 & 0 & 0\\
0 & 1/3 & 0 & 2/3 & 0\\
0 & 0 & 0 & 0 & 1
\end{pmatrix}.
\end{equation}

Markov maps preserve the $1$-norm and restrict to embeddings $\Delta_{m-1}^\circ\to\Delta_{n-1}^\circ$.

\begin{theorem}[Chentsov's theorem.]
 \mbox{}
 \label{thm:Chentsov}
\begin{itemize}
\item
Let $g^{(m)}$ be a Riemannian metric on $\Delta_{m-1}^{\circ}$ for $m\in \{ 2,3,\ldots \}$. Let this sequence of metrics have the property that every
congruent embedding by a Markov mapping %(Campbell map)
is an isometry.
Then, there is a constant $C>0$ that satisfies:
\begin{equation}
g_p^{(m)}(u, v ) = C \sum_{i}\frac{u_{i}v_{i}}{p_{i}}. \label{eq:Chentsov}
\end{equation}
\item
Conversely, for any $C>0$, the metrics given by Equation~\eqref{eq:Chentsov} define a sequence of Riemannian metrics under which every
congruent embedding by a Markov mapping %(Campbell map)
is an isometry.
\end{itemize}
\end{theorem}

The main result in Campbell's work~\cite{Campbell} is the following variant of Chentsov's theorem.

\begin{theorem}[Campbell's theorem.]
\mbox{}
 \label{thm:Campbell}
\begin{itemize}
\item
Let $g^{(m)}$ be a Riemannian metric on $\R^m_+$ for $m\in \{ 2,3,\ldots \}$. Let this sequence of metrics have the property that every embedding by a Markov mapping %Campbell map
is an isometry.
Then:
\begin{equation}
g_p^{(m)}(\partial_i, \partial_j ) =
A(|p|) + \delta_{ij}C(|p| ) \frac{|p|}{p_i}, \label{eq:Campbell}
\end{equation}
where $|p|=\sum_{i=1}^m p_i$, $\delta_{ij}$ is the Kronecker delta, and $A$ and $C$ are $C^\infty$ functions on $\R_+$ satisfying $C(\alpha) >0$ and $A(\alpha) + C(\alpha)>0$ for all $\alpha >0$.
\item
Conversely, if $A$ and $C$ are $C^\infty$ functions on $\R^+$ satisfying $C(\alpha)>0$, $A(\alpha) + C(\alpha) >0$ for all $\alpha>0$, then Equation~\eqref{eq:Campbell} defines a sequence of Riemannian metrics under which every embedding by a Markov mapping %Campbell map
is an isometry.
\end{itemize}
\end{theorem}

The metrics from Campbell's theorem also define metrics on the probability simplices $\Delta_{m-1}^\circ$ for $m=2,3,\ldots$.
Since the tangent vectors $v=\sum_{i}v_i\partial_i\in T_{p}\Delta_{m-1}^\circ$ satisfy $\sum_i v_i =0$, for any two vectors $u,v\in T_{p}\Delta_{m-1}^\circ$, also $\sum_{i}\sum_j A u_i v_j =0$ for any $A$. In this case, the choice of $A$ is immaterial, and the metric becomes Chentsov's metric.

\begin{remark}
\label{remark1}
Observe that Chentsov's theorem is not a direct implication of Campbell's theorem.
However, it can be deduced from it by the following arguments.
Suppose that we have a family of Riemannian simplices $(\Delta_{m-1}^\circ, g^{(m)})$ for $m\in\{2,3,\ldots\}$, and suppose that they are isometric with respect to Markov maps.
If we can extend every $g^{(m)}$ to a Riemannian metric $\tilde g^{(m)}$ on~$\R_{+}^{m}$ in such a way that the resulting spaces $(\R_+^{m},\tilde g^{(m)})$ are still isometric with respect to Markov maps,
then Campbell's theorem implies that $g^{(m)}$ is a multiple of the Fisher metric.
Such metric extensions can be defined as follows.
Consider the diffeomorphism:
\begin{equation}
 \Delta_{m-1}^{\circ}\times\R_{+} \cong \R_{+}^{m}, \quad
 (p,r) \mapsto r\cdot p.
\end{equation}
Any tangent vector $u\in T_{(p,r)}\R_{+}^{m}$ can be written uniquely as $u = u_{p} + u_{r}\partial_{r}$, where $u_{p}$ is tangent to $r\Delta_{m-1}^\circ$.
Since each Markov map~$f$ preserves the one-norm~$|\cdot|$, its push-forward
$f_\ast$ maps the tangent vector $\partial_{r}\in T_{(p,r)}\R_{+}^{m}$ to the corresponding tangent vector
$\partial_{r}\in T_{f(p,r)}\R_{+}^{m}$; that is, $f_\ast u = f_\ast u_{p} + u_{r}\partial_{r}$.
Therefore,
\begin{equation}
 \tilde g^{(m)}_{(p,r)}(u,v) := g^{(m)}_{p}(u_{p},v_{p}) + u_{r}v_{r}
\end{equation}
is a metric on $\R_{+}^{m}$ that is invariant under~$f$.
\end{remark}

In what follows, we will focus on positive matrices.
In order to define a natural Riemannian metric, we can use the identification $\R_+^{k\times m}\cong \R_+^{k m}$ and apply Campbell's theorem.
This leads to metrics of  the form:
\begin{equation}
 \label{eq:Campbell-for-matrices}
g^{(k,m)}_M(\partial_{ij},\partial_{kl}) = A(|M|) + \delta_{ik}\delta_{jl} C(|M|) / M_{ij},
\end{equation}
where $\partial_{ij}=\frac{\partial}{\partial M_{ij}}$ and $|M|=\sum_{ij}M_{ij}$.
However, a disadvantage of this approach is that the action of general Markov maps on $\R_{+}^{k m}$ has no natural interpretation in terms of the matrix structure.
Therefore, Lebanon~\cite{Lebanon05} considered a special class of Markov maps defined as follows. % which take into account the matrix structure.

\begin{definition}\label{def:lebanonmorphisms}
Consider a $k\times l$ row-partition matrix $R$
and a collection of $m\times n$ row-partition matrices $Q = \{ Q^{(1)}, \ldots, Q^{(k)} \}$.
The map:
\begin{equation}
 \R^{k\times m}_+ \to \R^{l\times n}_+;\quad M\mapsto R^\top ( M\otimes Q)
\end{equation}
is called a congruent embedding by a Markov morphism of $\R^{k\times m}_+$ to $\R^{l\times n}_+$ in~\cite{Lebanon04}.
We will refer to such an embedding as a Lebanon map.
Here, the row product $M\otimes Q$ is defined by:
\begin{equation}
 (M\otimes Q)_{ab} = (M\cdot Q^{(a)})_{ab}, \quad\text{for all $a\in[k], b\in[n]$};
\end{equation}
that is, the $a$-th row of $M$ is multiplied by the matrix~$Q^{(a)}$.
\end{definition}
In a Lebanon map, each row of the input matrix $M$ is mapped by an individual Markov mapping $Q^{(i)}$,
and each resulting row is copied and scaled by an entry of $R$.
This kind of map preserves the sum of all matrix entries.
Therefore, with the identification $\R_+^{k\times m}\cong \R_{+}^{km}$, each Lebanon map restricts to a map $\Delta_{mk-1}^\circ\to\Delta_{nl-1}^\circ$.
The set $\Delta_{mk-1}^\circ$ can be identified with the set of joint distributions of two random variables.
Lebanon maps can be regarded as special Markov maps that incorporate the product structure present in the set of joint probability distributions of a pair of random variables.
In Section~\ref{section:invariance}, we will give an interpretation of these maps.

Contrary to what is stated in~\cite{Lebanon04}, a Lebanon map does not map $(\Delta_{m-1}^k)^\circ$ to $(\Delta_{n-1}^l)^\circ$, unless $k=l$.
Therefore, later, we will provide a characterization for the metrics on $(\Delta_{m-1}^k)^\circ$ in terms of invariance under other maps (which are not Markov nor Lebanon maps).

The main result in Lebanon's work~\cite[Theorems~1 and 2]{Lebanon04} is the following.

\begin{theorem}[Lebanon's theorem.]
\mbox{}
 \label{thm:Lebanon}
 \begin{itemize}
 \item For each $k\ge 1, m\ge 2$, let $g^{(k,m)}$ be a
 Riemannian metric on $\R^{k\times m}_{+}$ in such a way that
 every Lebanon map is an isometry. Then:
 \begin{equation}
  \label{eq:Lebanon-metrics}
  g_M^{(k,m)}(\partial_{ab}, \partial_{cd})
  = A(|M|) +\delta_{ac} \left( \frac{B(|M|)}{|M_a|} + \delta_{bd} \frac{C(|M|)}{M_{ab}} \right)
 \end{equation}
 for some differentiable functions $A, B, C \in C^\infty(\R_+)$.
 \item
 Conversely, let $\{ (\R^{k\times m}_+, g^{(k,m)} ) \}$ be a sequence of Riemannian manifolds, with metrics $g^{(k,m)}$ of the
 form~\eqref{eq:Lebanon-metrics} for some $A, B, C \in C^{\infty}(\R_+)$. Then, every Lebanon map is an isometry.
\end{itemize}
\end{theorem}
Lebanon does not study the question under which assumptions on $A,B,C\in C^{\infty}(\R_{+})$
the formula~\eqref{eq:Lebanon-metrics} does indeed define a Riemannian metric.
This question has the following simple answer, which we
will prove in Appendix~\ref{section:positive}:
\begin{proposition}
 \label{proposition:positivedef}
 The matrix~\eqref{eq:Lebanon-metrics} is positive definite if and only if $C(|M|)>0$, $B(|M|)+C(|M|)>0$ and
 $A(|M|)+B(|M|)+C(|M|)>0$.
\end{proposition}

The class of metrics~\eqref{eq:Lebanon-metrics} is larger than the class of metrics~\eqref{eq:Campbell-for-matrices} derived in Campbell's theorem.
The reason is that Campbell's metrics are invariant with respect to a larger class of embeddings.

The special case with $A(|M|) = 0$, $B(|M|) = 0$ and $C(|M|) = 1$ is called product Fisher metric,
\begin{equation}
g^{(k,m)}_M (\partial_{ab},\partial_{cd}) = \delta_{ac}\delta_{bd} \frac{1}{M_{ab}}.
\label{eq:productFishermetric}
\end{equation}
Furthermore, if we restrict to $(\Delta_{m-1}^{k})^\circ$,
the functions $A$ and $B$ do not play any role.
In this case $|M|=k$, and we obtain the scaled product Fisher metric:
\begin{equation}
 \label{eq:Lebanon-normalized}
g^{(k,m)}_M (\partial_{ab},\partial_{cd}) = \delta_{ac}\delta_{bd} \frac{C(k)}{M_{ab}},
\end{equation}
where $C(k):\mathbb{N}\to\R_{+}$ is a positive function.
As mentioned before, Lebanon's theorem does not give a characterization of invariant metrics of stochastic matrices, since Lebanon maps do not preserve the stochasticity of the matrices. However, Lebanon maps are natural maps on the set~$\Delta_{mk-1}^\circ$ of positive joint distributions. In the same way as Chentsov's theorem can be derived from Campbell's theorem (see Remark~\ref{remark1}), we obtain the following corollary:

\begin{corollary}
 \mbox{}
 \label{cor:Lebanon}
 \begin{itemize}
 \item Let $\{ ( \Delta_{k m-1}^{\circ} , g^{(k,m)} ) \colon k\geq 1, m\geq 2 \}$ be a double sequence of Riemannian
 manifolds with the property that every Lebanon map is an isometry. Then:
 \begin{equation}
  \label{eq:Lebanon-joint-metrics}
  g_P^{(k,m)}(u, v) % u = \sum_{a,b} u_{a,b}\partial_{a,b}
  = B \sum_{a}\sum_{b,c}\frac{u_{ab}u_{ac}}{|P_a|} + C \sum_{a}\sum_{b}\frac{u_{ab}v_{ab}}{P_{ab}}, \quad\text{for each $P\in\Delta_{k m-1}^{\circ}$},
 \end{equation}
 for some constants $B, C\in\R$ with $C>0$ and $B+C>0$, where $|P_a|=\sum_b P_{ab}$.
 \item
 Conversely, let $\{ (\Delta_{km-1}^{\circ} , g^{(k,m)} ) \}$ be a sequence of Riemannian manifolds with metrics $g^{(k,m)}$ of the
 form of~Equation \eqref{eq:Lebanon-joint-metrics} for some $B, C \in \R$. Then, every Lebanon map is an isometry.
\end{itemize}
\end{corollary}
Observe that these metrics agree with (a multiple of) the Fisher metric only if~$B=0$. The case $B=0$ can also be
characterized;
note that Lebanon maps do not treat the two random variables symmetrically. Switching the two
random variables corresponds to transposing the joint distribution matrix~$P$. When exchanging the role of the two
random variables, the Lebanon map becomes $P \mapsto (P^{\top}\otimes Q)^{\top} R$. We call such a map a dual Lebanon map.
 If we require invariance under both Lebanon maps and their duals in Theorem~\ref{thm:Lebanon} or
Corollary~\ref{cor:Lebanon}, the statements remain true with the additional restriction that $B=0$ (as a function or constant, respectively).

\section{Invariance Metric Characterizations for Conditional Polytopes}\label{section:invariance}

According to Chentsov's theorem (Theorem~\ref{thm:Chentsov}),
a natural metric on the probability simplex can be characterized by requiring the isometry of natural embeddings.
Lebanon follows this axiomatic approach to characterize metrics on products of positive measures (Theorem~\ref{thm:Lebanon}).
However, the maps considered by Lebanon dissolve the row-normalization of conditional distributions.
In general, they do not map conditional polytopes to conditional polytopes. Therefore, we will consider a slight modification of Lebanon maps,
in order to obtain maps between conditional polytopes.

\subsection{Stochastic Embeddings of Conditional Polytopes}
\label{sec:stochsticembeddings}

A matrix of conditional distributions $P(Y|X)$ in $\Delta_{m-1}^k$ can be regarded as the equivalence class of all joint probability distributions $P(X,Y) \in \Delta_{km-1}$ with conditional distribution $P(Y|X)$.
Which Markov maps of probability simplices are compatible with this equivalence relation?
The most obvious examples are permutations (relabelings) of the state spaces of~$X$ and~$Y$.

In information theory, stochastic matrices are also viewed as channels.
For any distribution of $X$, the stochastic matrix gives us a joint distribution of the pair $(X,Y)$ and, hence, a marginal distribution of~$Y$. If we input a distribution of~$X$ into the channel, the stochastic matrix determines what the distribution of the output~$Y$ will be.

Channels can be combined, provided the cardinalities of the state spaces fit together. If we take the output $Y$ of the
first channel $P(Y|X)$ and feed it into another channel $P(Y'|Y)$ then we obtain a combined channel~$P(Y'|X)$.
The composition of channels corresponds to ordinary matrix multiplication.
If the first channel is described by the stochastic matrix $K$ and the second channel by~$Q$, then the combined channel is described by~$K\cdot Q$. Observe that in this case, the joint distribution $P$ (considered as a normalized matrix $P\in\Delta_{km-1}$) is transformed similarly; that is, the joint distribution of the pair $(X,Y')$ is given by~$P\cdot Q$.

More general maps result from compositions where the choice of the second channel depends on the input of the first channel.
In other words, we have a first channel that takes as input $X$ and gives as output~$Y$, and we have another channel that takes as input $(X, Y)$ and gives as output $Y'$; we are interested in the resulting channel from $X$ to~$Y'$.
The second channel can be described by a collection of
stochastic matrices $Q=\{Q^{(i)} \}_{i}$.
If $K$ describes the first channel, then the combined channel is described by the row product~$K\otimes Q$ (see Definition~\ref{def:lebanonmorphisms}). Again, the joint distribution of $(X,Y')$ arises in a similar way as~$P\otimes Q$.

We can also consider transformations of the first random variable~$X$. Suppose that
we use $X$ as the input to a channel described by a stochastic matrix~$R$.
In this case, the joint distribution of the output $X'$ of the channel and $Y$ is described by~$R^{\top}X$.
However, in general, there is not much that we can say about the conditional distribution of~$Y$ given~$X'$.
The result depends in an essential way on the original distribution of~$X$. However,
this is not true in the special case that the channel is ``not mixing'', that is, in the case that $R$ is a stochastic partition matrix.
In this case, the conditional distribution $P(Y|X')$ is described by~$\ol{R}^{\top}K$, where $\ol R$ is the corresponding partition
indicator matrix, where all non-zero entries of $R$ are replaced by~one. In other words, each state of~$X$ corresponds to
several states of~$X'$, and the corresponding row of~$K$ is copied a corresponding number of times.

To sum up, if we combine the transformations due to $Q$ and~$R$, then the joint probability distribution transforms as
$P\mapsto R^{\top}(P\otimes Q)$ and the conditional transforms as $K\mapsto \ol R^{\top}(K\otimes Q)$.
In particular, for the joint distribution, we obtain the definition of a Lebanon map. Figure~\ref{fig:Lebanon-maps} illustrates the
situation.

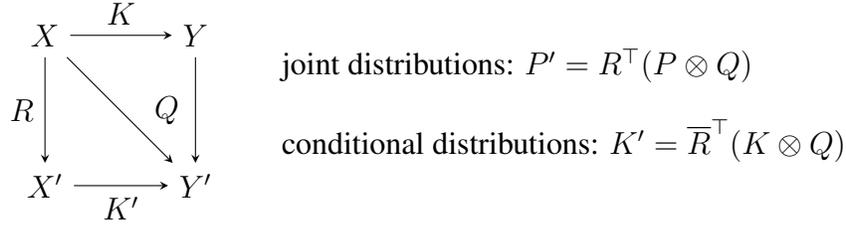
\begin{figure} %[H]
 \centering
 \begin{tikzpicture}
 \node (X) at (0,2) {$X$};
 \node (Y) at (2,2) {$Y$};
 \node (X') at (0,0) {$X'$};
 \node (Y') at (2,0) {$Y'$};
 \begin{scope}[->,>=stealth]
  \path (X) edge node[above] {$K$} (Y);
  \path (X') edge node[below] {$K'$} (Y');
  \path (X) edge node[left] {$R$} (X');
  \path (X) edge (Y');
  \path (Y) edge node[left] {$Q$\,} (Y');
 \end{scope}
 \node[anchor=west] at (3,1.6) {joint distributions: $P' = R^\top(P\otimes Q)$};
 \node[anchor=west] at (3,0.6) {conditional distributions: $K' = \ol{R}^\top(K\otimes Q)$};
 \end{tikzpicture}
 \caption{An interpretation for Lebanon maps and conditional embeddings. The variable $X'$ is computed from $X$
 by~$R$, and $Y'$ is computed from $X$ and $Y$ by~$Q$.}
 \label{fig:Lebanon-maps}
\end{figure}

Finally, we will also consider the special case where the partition of $R$ (and $\ol R$) is homogeneous, \emph{i.e.}, such that all blocks have the same size.
For example, this describes the case where there is a third random variable $Z$ that is independent of~$Y$ given~$X$. In this case, the conditional distribution satisfies $P(Y|X) = P(Y|X,Z)$, and $R$ describes the conditional distribution of $(X,Z)$ given~$X$.

\begin{definition}
A (row) partition indicator matrix is a matrix $\ol R\in\{0,1\}^{k\times l}$ that satisfies:
\begin{equation}
 \ol R_{ij} =
 \begin{cases}
 1, & \text{ if }j\in A_{i}, \\
 0, & \text{ else},
 \end{cases}
\end{equation}
for a $k$ block partition $\{A_1,\ldots, A_k \}$ of~$[l]$.
\end{definition}

For example, the $3\times 5$ partition indicator matrix corresponding to~Equation \eqref{eq:stoch-part-mat} is:
\begin{equation}
\overline{R} = \begin{pmatrix}
1 & 0 & 1 & 0 & 0 \\
0 & 1 & 0 & 1 & 0 \\
0 & 0 & 0 & 0 & 1
\end{pmatrix}.
\end{equation}

\begin{definition}\label{def:lebanonmodi}
Consider a $k\times l$ partition indicator matrix $\ol R$
and a collection of $m\times n$ stochastic partition matrices $Q=\{Q^{(i)}\}_{i=1}^k$.
We call the map:
\begin{equation}
f\colon\quad \R^{k\times m}_+ \to \R^{l\times n}_+;\quad M\mapsto \ol R^\top ( M\otimes Q)
\end{equation}
a conditional embedding of $\R^{k\times m}_+$ in $\R^{l\times n}_+$.
We denote the set of all such maps by $\hat\Fcal_{k,m}^{l,n}$.
If $\ol R$ is the partition indicator matrix of a homogeneous partition (with partition blocks of equal cardinality),
then we call $f$ a homogeneous conditional embedding.
We denote the set of all such homogeneous conditional embeddings by $\Fcal_{k,m}^{l,n}$ and assume that $l$ is a multiple of $k$.
\end{definition}

Conditional embeddings preserve the $1$-norm of the matrix rows; that is, the elements of $\hat\Fcal_{k,m}^{l,n}$ map $(\Delta^{k}_{m-1})^\circ$ to~$(\Delta^{l}_{n-1})^\circ$.
On the other hand, they do not preserve the $1$-norm of the entire matrix.
Conditional embeddings are Markov maps only when $k=l$, in which case they are also Lebanon maps.

\subsection{Invariance Characterization}

Considering the conditional embeddings discussed in the previous section, we obtain the following metric characterization.

\begin{theorem}
\label{thm:isochar}
\mbox{}
\begin{itemize}
\item
Let $g^{(k,m)}$ denote a metric on $\R_+^{k\times m}$ for each $k\geq 1$ and $m\geq 2$.
If every homogeneous conditional embedding $f\in\mathcal{F}_{k,m}^{l,n}$
is an isometry with respect to these metrics, then:
\begin{equation}
g_{M}^{(k,m)} (\partial_{ab},\partial_{cd}) =
\frac{A}{k^2} + \delta_{ac}\left( k\frac{B}{k^2} + \delta_{bd} \frac{ |M| }{ M_{ab}} \frac{C}{k^2} \right), \quad\text{for all $M\in \R_+^{k\times m}$},
 \label{eq:invmetric}
\end{equation}
for some constants $A, B, C\in\R$, where $\partial_{ab}=\frac{\partial}{\partial M_{ab}}$ and $|M| =\sum_{ab}M_{ab}$.
\item
Conversely, given the metrics defined by Equation \eqref{eq:invmetric} for any non-degenerate choice of constants $A,B,C\in\R$,
each homogeneous conditional embedding $f\in\mathcal{F}_{k,m}^{l,n}$, $k\leq l$, $m\leq n$ is an isometry.

\item
Moreover, the tensors $g^{(k,m)}$ from Equation~\eqref{eq:invmetric} are positive-definite for all $k\geq 1$ and $m\geq 2$ if and only if $C>0$, $B+C>0$ and $A+B+C>0$.
\end{itemize}
\end{theorem}
The proof of Theorem~\ref{thm:isochar} is similar to the proof of the theorems of Chentsov, Campbell and Lebanon. Due
to its technical nature, we defer it to Appendix~\ref{sec:proofs-invariance}.

Now, for the restriction of the metric $g^{(k,m)}$ to $(\Delta_{m-1}^{k})^\circ$, we have the following.
In this case, $|M|=k$.
Since tangent vectors $v=\sum_{ab}v_{ab}\partial_{ab}\in T_{M}(\Delta_{m-1}^{k})^\circ$ satisfy $\sum_{b}v_{ab}=0$ for all $a$,
the constants $A$ and $B$ become immaterial, and the metric can be written as:
\begin{equation}
\label{eq:invmetric-norm}
g_{M}^{(k,m)} (u,v)
= \sum_{ab} \frac{ |M| u_{ab}v_{ab} }{ M_{ab}} \frac{C}{k^2}
= \sum_{ab} \frac{ u_{ab}v_{ab} }{ M_{ab}} \frac{C}{k}, \quad\text{for all $u,v\in T_M(\Delta_{m-1}^k)^\circ$}.
\end{equation}
This metric is a specialization of the metric~\eqref{eq:Lebanon-normalized} derived by Lebanon (Theorem~\ref{thm:Lebanon}).

The statement of Theorem~\ref{thm:isochar} becomes false if we consider general conditional embeddings instead of
homogeneous ones:
\begin{theorem}\label{thm:isocharstrong}
 There is no family of metrics $g^{(k,m)}$ on $\R_+^{k\times m}$ (or on $(\Delta_{m-1}^{k})^\circ$)
 for each $k\geq 1$ and $m\geq 2$,
 for which every conditional embedding $f\in\hat\Fcal_{k,m}^{l,n}$ is an isometry.
\end{theorem}
This negative result will become clearer from the perspective of Section~\ref{sec:weighted-metrics}: as we will show in
Theorem~\ref{thm:Kakade-characterization}, although there are no metrics that are invariant under all conditional
embeddings, there are families of metrics (depending on a parameter,~$\rho$) that transform covariantly (that is,
in a well-defined manner) with respect to the conditional embeddings.
We defer the proof of Theorem~\ref{thm:isocharstrong} to Appendix~\ref{sec:proofs-invariance}.

\section{The Fisher Metric on Polytopes and Point Configurations}
\label{sec:polytopes-Fisher}

In the previous section, we obtained distinguished Riemannian metrics on $\R_+^{k\times m}$ and $(\Delta_{m-1}^{k})^\circ$ by postulating invariance
under natural maps.
In this section, we take another viewpoint based on general considerations about Riemannian metrics on arbitrary polytopes.
This is achieved by embedding each polytope in a probability simplex as an exponential family.
We first recall the necessary background.
In Section~\ref{sec:invariance-polytopes}, we then present our general results, and in Section~\ref{sec:independence-models}, we discuss the special case of conditional polytopes.

\subsection{Exponential Families and Polytopes}\label{section:expfams}

Let $\Xcal$ be a finite set and $A\in\R^{d\times\Xcal}$ a matrix with columns $a_{x}$ indexed by~$x\in\Xcal$.
It will be convenient to consider the rows $A_{i}$, $i\in[d]$ of $A$ as functions
$A_{i}:\Xcal\to\R$.
Finally, let $\nu\colon \Xcal\to \R_+$.
The exponential family $\Ecal_{A,\nu}$ is the set of probability distributions on $\Xcal$ given by:
\begin{equation}
 p(x;\theta) = \exp(\theta^\top a_{x} + \log ( \nu(x) ) - \log ( Z(\theta)) ), \quad\text{for all $x\in\Xcal$}, \quad\text{for all $\theta\in\R^d$},
\end{equation}
with the normalization function $Z(\theta) = \sum_{x'\in\Xcal} \exp(\theta^\top a_{x'} + \log ( \nu(x')) )$.
The functions $A_i$ are called the observables and $\nu$ the reference measure of the exponential family.
When the reference measure $\nu$ is constant, $\nu(x)=1$ for all~$x\in\Xcal$, we omit the subscript and write $\Ecal_A$.

A direct calculation shows that
the Fisher information matrix of $\Ecal_{A,\nu}$ at a point $\theta\in\R^d$   has coordinates:
\begin{align}
 g^{\Ecal_{A,\nu}}_\theta(\partial_{\theta_{i}}, \partial_{\theta_{j}})
&= \cov_\theta(A_i, A_j), \qquad\text{for all $i,j\in[d]$}.
\end{align}
Here, $\cov_\theta$ denotes the covariance computed with respect to the probability distribution $p(\cdot;\theta)$.

The convex support of $\Ecal_{A,\nu}$ is defined as:
\begin{equation}
\conv A :
 = \conv\{ a_x \colon x\in\Xcal \}
 = \Big\{\E_p[ A] \colon p\in \Delta_{|\Xcal|-1} \Big\}
 = \Big\{\E_p[ A] \colon p\in\overline{\Ecal_{A,\nu}}\Big\},
\end{equation}
where $\conv S$ is the set of all convex combinations of points in $S$.
The moment map $\mu: p\in\Delta_{n-1}\mapsto A\cdot p\in\R^{d}$ restricts to a homeomorphism~$\ol{\Ecal_{A,\nu}}\to\conv A$; see~\cite{barndorff1978information}.
Here, $\overline{\Ecal_{A,\nu}}$ denotes the Euclidean closure of~$\Ecal_{A,\nu}$.
The inverse of $\mu$ will be denoted by
$\mu^{-1} :\conv A \to\ol{\Ecal_{A,\nu}}\subseteq\Delta_{n-1}$.
This gives a natural embedding of the polytope $\conv A$ in the probability simplex $\Delta_{|\Xcal|-1}$.
Note that the convex support is independent of the reference measure $\nu$.
See~\cite{Brown1986} for more details.

\subsection{Invariance Fisher Metric Characterizations for Polytopes}
\label{sec:invariance-polytopes}

Let $\Pbf\in\R^{d}$ be a polytope with $n$ vertices $a_{1},\dots,a_{n}$. Let $A=(a_{1},\dots,a_{n})$ be the matrix with
columns~$a_{i}\in \R^d$ for all $i\in[n]$. Then, $\Ecal_{A}\subseteq\Delta_{n-1}^\circ$ is an exponential family with convex support~$\Pbf$.
We will also denote this exponential family by~$\Ecal_{\Pbf}$.
We can use the inverse of the moment map, $\mu^{-1}$, to pull back geometric structures
on~$\Delta_{n-1}^\circ$ to the relative interior~$\Pbf^\circ$ of $\Pbf$.
\begin{definition}
\label{definition:fishexp}
 The Fisher metric on $\Pbf^{\circ}$ is the pull-back
 of the Fisher metric on $\Ecal_{A}\subseteq\Delta_{n-1}^\circ$ by~$\mu^{-1}$.
\end{definition}

Some obvious questions are: Why is this a natural construction? Which maps between polytopes are isometries between
their Fisher metrics? Can we find a characterization of Chentsov type for this metric?

Affine maps are natural maps between polytopes. However, in order to obtain isometries, we need to put some additional constraints.
Consider two polytopes $\Pbf\in\R^{d}$, $\Pbf'\in\R^{d'}$ and an affine
map~$\phi:\R^{d}\to\R^{d'}$ that satisfies $\phi(\Pbf)\subseteq\Pbf'$.
A natural condition in the context of exponential families is that $\phi$ restricts to a bijection between the set $\ver(\Pbf)$ of vertices
of~$\Pbf$ and the set $\ver(\Pbf')$ of vertices of~$\Pbf'$.
In this case, $\Ecal_{\Pbf'}\subseteq\Ecal_{\Pbf}\subseteq\Delta_{n-1}^\circ$. Moreover, the moment map $\mu'$ of~$\Pbf'$ factorizes through
the moment map $\mu$ of~$\Pbf$: $\mu'=\phi\circ\mu$. Let $\phi^{-1} = \mu\circ\mu^{\prime-1}$.
Then, the following diagram commutes:
\begin{equation}
 \label{eq:cd-simple-morphisms}
 \begin{tikzpicture}[baseline=0pt]
 \node (P) at (0,1) {$\Pbf^\circ$};
 \node (E) at (2,1) {$\Ecal_{\Pbf}$};
 \node (D) at (4,0) {$\Delta_{n-1}^\circ$};
 \node (Pp) at (0,-1) {$\Pbf'^\circ$};
 \node (Ep) at (2,-1) {$\Ecal_{\Pbf'}$};
 \begin{scope}[->,>=stealth]
  \path (P) edge node[above] {$\mu^{-1}$} (E);
  \path[right hook->] (E) edge (D);
  \path (Pp) edge node[below] {$\mu^{\prime-1}$} (Ep);
  \path[right hook->] (Ep) edge (D);
  \path (Pp) edge node[left] {$\phi^{-1}$} (P);
  \path[right hook->] (Ep) edge (E);
 \end{scope}
 \end{tikzpicture}
\end{equation}
It follows that $\phi^{-1}$ is an isometry from $\Pbf^{\prime\circ}$ to its image in~$\Pbf^{\circ}$.
Observe that the inverse moment map itself arises in this way: In the diagram~\eqref{eq:cd-simple-morphisms}, if $\Pbf$ is equal to $\Delta_{n-1}$, then the upper moment map $\mu^{-1}$ is the identity map, and $\phi^{-1}$ equals the inverse moment map $\mu^{\prime-1}$ of~$\Pbf'$.

The constraint of mapping vertices to vertices bijectively is very restrictive.
In order to consider a larger class of affine maps, we need to generalize our construction from polytopes to weighted  point configurations.
\begin{definition}
 A weighted point configuration is a pair $(A,\nu)$ consisting of a
 matrix $A\in\R^{d\times n}$ with columns $a_{1},\dots,a_{n}$
 and a positive weight function $\nu:\{1,\dots,n\}\to\R_{+}$ assigning a weight to each column $a_i$.
 The pair $(A,\nu)$ defines the exponential family $\Ecal_{A,\nu}$.

 The $(A,\nu)$-Fisher metric on~$(\conv A)^\circ$ is the pull-back of the Fisher metric on~$\Delta_{n-1}^\circ$ through the inverse
 of the moment map.
\end{definition}
We recover Definition~\ref{definition:fishexp} as follows.
For a polytope $\Pbf$, let $A$ be the point configuration consisting of
the vertices of~$\Pbf$.
Moreover, let $\nu$ be a constant function. Then, $\Ecal_{\Pbf}=\Ecal_{A,\nu}$, and the two definitions of the Fisher metric on~$\Pbf^{\circ}$ coincide.

The following are natural maps between weighted point configurations:
\begin{definition}
 Let $(A,\nu)$, $(A',\nu')$ be two weighted point configurations with $A=(a_{i})_{i}\in\R^{d\times n}$ and
   $A'=(a'_{j})_{j}\in\R^{d'\times n'}$. A morphism $(A,\nu)\to(A',\nu')$ is a pair $(\phi,\sigma)$ consisting of
 an affine map  $\phi:\R^{d}\to\R^{d'}$ and a surjective map $\sigma:\{1,\dots,n\}\to\{1,\dots,n'\}$ with $\phi(a_{i}) =
 a'_{\sigma(i)}$ and  $\nu'(a'_{j}) = \alpha\sum_{i:\sigma(i)=j}\nu(a_{i})$, where $\alpha>0$ is a constant that does
 not depend on~$j$.
\end{definition}

Consider a morphism $(\phi,\sigma):(A,\nu)\to(A',\nu')$. For each $j\in[n']$, let $A_{j}=\{i:\phi(a_{i})=a'_{j}\}$.
Then, $(A_{1},\dots,A_{n'})$ is a partition of~$[n]$. Define a matrix $Q\in\R^{n'\times n}$ by:
\begin{equation}
 Q_{ji} =
 \begin{cases}
 \frac{\nu(i)}{\sum_{i'\in A_{j}}\nu(i')}, & \text{ if }i\in A_{j}, \\
 0, & \text{ else}.
 \end{cases}
\end{equation}
Then, $Q$ is a Markov mapping, and the following diagram commutes:
\begin{equation}
 \label{eq:comm-diag}
 \begin{tikzpicture}[baseline=0pt]
 \node (P) at (-0.5,1) {$(\conv A)^\circ$};
 \node (E) at (2,1) {$\Ecal_{A,\nu}$};
 \node (D) at (4,1) {$\Delta_{n-1}^\circ$};
 \node (Pp) at (-0.5,-1) {$(\conv A')^\circ$};
 \node (Ep) at (2,-1) {$\Ecal_{A',\nu'}$};
 \node (Dp) at (4,-1) {$\Delta_{n'-1}^\circ$};
 \begin{scope}[->,>=stealth]
  \path (P) edge node[above] {$\mu^{-1}$} (E);
  \path[right hook->] (E) edge (D);
  \path (Pp) edge node[below] {$\mu^{\prime-1}$} (Ep);
  \path[right hook->] (Ep) edge (Dp);
  \path (Pp) edge node[left] {$\phi^{-1}$} (P);
  \path[right hook->] (Ep) edge (E);
  \path[right hook->] (Dp) edge node[right] {$Q$} (D);
 \end{scope}
 \end{tikzpicture}
\end{equation}
By Chentsov's theorem (Theorem~\ref{thm:Chentsov}),
$Q$ is an isometric embedding. It follows that $\phi^{-1}$ also induces an isometric embedding.
This shows the first part of the following theorem:
\begin{theorem}
\mbox{ }
 \label{thm:metric-on-polytopes}
 \begin{itemize}
 \item
 Let $(\phi,\sigma):(A,\nu)\to(A',\nu')$ be a morphism of weighted point configurations. Then, $\phi^{-1}:(\conv A')^\circ\to (\conv A)^\circ$ is an isometric embedding with respect to the Fisher metrics on $(\conv A)^{\circ}$ and
 $(\conv A')^{\circ}$.
 \item
 Let $g^{A,\nu}$ be a Riemannian metric on $(\conv A)^{\circ}$ for each weighted point configuration $(A,\nu)$.
 If every morphism $(\phi,\sigma):(A,\nu)\to(A',\nu')$ of weighted point configurations induces an isometric embedding
 $\phi^{-1}:(\conv A')^\circ \to (\conv A)^\circ$, then there exists a constant $\alpha\in\R_{+}$ such that $g^{A,\nu}$ is equal to
 $\alpha$ times the $(A,\nu)$-Fisher metric.
 \end{itemize}
\end{theorem}
\begin{proof}
 The first statement follows from the discussion before the theorem.
 For the second statement, we show that under the given assumptions, all Markov maps are isometric embeddings.
 By Chentsov's theorem (Theorem~\ref{thm:Chentsov}), this implies that the metrics $g^{\Pbf}$ agree with the Fisher metric whenever $\Pbf$ is a simplex.
 The statement then follows from the two facts that the metric on~$\Pbf^\circ$ or $(\conv A)^\circ$ is the   pull-back of the Fisher metric through the inverse of the moment map and that $\mu^{-1}$ is itself a morphism.

 Observe that $\Delta_{n-1} = \conv I_n = \conv\{ e_{1},\dots,e_{n}\}$ is a polytope, and $\Delta_{n-1}^{\circ}$ is the corresponding exponential family.
 Consider a Markov embedding $Q:\Delta_{n'-1}^\circ\to\Delta_{n-1}^\circ, p\mapsto p \cdot Q$.
 Let $\nu(i)=\sum_{j}Q_{ji}$ be the value of the unique non-zero entry of~$Q$ in the $i$-th column.
 This defines a morphism and an embedding as follows:

 Let $A$ be the matrix that arises from~$Q$ by replacing each non-zero entry by~one.
 We define $\phi$ as the linear map represented by the matrix~$A$,
 and define $\sigma:[n]\to[n']$ by
$\sigma(j) = i$ if and only if $a_{j}=e_{i}$, that is, $\sigma(j)$ indicates the row $i$ in which the $j$-th column of~$A$ is non-zero.
 Then, $(\phi,\sigma)$ is a morphism
 $(I_n, \nu)\to(I_{n'}, 1)$,
 and by assumption, the inverse $\phi^{-1}$ is an isometric
 embedding $\Delta_{n'-1}^\circ \to \Delta_{n-1}^\circ$.
 However, $\phi^{-1}$ is equal to the Markov map~$Q$.
 This shows that all Markov maps are isometric embeddings, and so, by Chentsov's theorem, the statement holds true on the simplices.
\end{proof}

Theorem~\ref{thm:metric-on-polytopes} defines a natural metric on~$(\Delta_{m-1}^{k})^\circ$ that we want to discuss in more detail next.

\subsection{Independence Models and Conditional Polytopes}
\label{sec:independence-models}

Consider $k$ random variables with finite state spaces $[n_1], \ldots, [n_k]$. The independence model consists
of all joint distributions $p\in\Delta_{\prod_{i\in[k]} n_i-1}$ of these variables that factorize as:
\begin{equation}
 \label{eq:param-indepmod}
 p(x_1,\ldots, x_k) = \prod_{i\in[k]} p_i(x_i), \quad\text{for all $x_1\in[n_1],\ldots, x_k\in[n_k] $},
\end{equation}
where $p_{i}\in\Delta_{n_{i}-1}$ for all $i\in[k]$.
Assuming fixed $n_1,\ldots, n_k$, we denote the independence model by $\ol{\Ecal_{k}}$. It is the Euclidean closure of an
exponential family (with observables of the form $\delta_{i y_i}$).
The convex support of $\Ecal_{k}$ is equal to the product of simplices $\Pbf_{k}:=\Delta_{n_1-1}\times \cdots\times\Delta_{n_k-1}$.
The parametrization~\eqref{eq:param-indepmod} corresponds to the inverse of the moment map.

We can write any tangent vector $u\in T_{(p_1,\ldots, p_k)}\Pbf_{k}^\circ$ of this open product of simplices as a linear combination $u=\sum_{i\in[k]}\sum_{x_i\in [n_i]} u_{i x_i} \partial_{i,x_i}$, where $\sum_{x_i \in [n_i]}v_{i x_i} =0$ for all $i\in[k]$.
Given two such tangent vectors, % $u,v\in T_{(p_1,\ldots, p_k)} \Pbf_k^\circ$,
the Fisher metric is given by:
\begin{equation}
 \label{eq:indep-product}
 g^{\Pbf_k}_{(p_1,\ldots, p_k)}(u,v)
 = \sum_{i\in[k]} \sum_{x_i \in [n_i]} \frac{ u_{i x_i} v_{i x_i} }{ p_i(x_i) }.
\end{equation}

Just as the convex support of the independence model is the Cartesian product of probability simplices, the Fisher metric on the independence model is the product metric of the Fisher metrics on the probability simplices of the individual variables.
If $n_{1}=\dots=n_{k}=:n$, then $\Pbf_{k}=\Delta^{k}_{n-1}$ can be identified with the set of $k\times n$ stochastic matrices.

The Fisher metric on the product of simplices is equal to the product of the Fisher metrics on the factors.
More generally, if $\Pbf=\Qbf_{1}\times\Qbf_{2}$ is a Cartesian product, then the Fisher metric on $\Pbf^\circ$ is equal to
the product of the Fisher metrics on $\Qbf_{1}^\circ$ and~$\Qbf_{2}^\circ$. In fact, in this case, the inverse of the moment map
of~$\Pbf$ can be expressed in terms of the two moment map inverses
$\mu_{1}:\Qbf_{1}\to\ol{\Ecal_{\Qbf_{1}}}\subseteq\Delta_{m_{1}-1}$ and
$\mu_{2}:\Qbf_{2}\to\ol{\Ecal_{\Qbf_{2}}}\subseteq\Delta_{m_{2}-1}$ and the moment map $\tilde\mu$ of the independence
model $\Delta_{m_{1}-1}\times\Delta_{m_{2}-1}$, by:
\begin{equation}
 \mu^{-1}(q_{1},q_{2}) = \tilde\mu^{-1}(\mu_{1}^{-1}(q_{1}), \mu_{2}^{-1}(q_{2})).
\end{equation}
Therefore, the pull-back by $\mu^{-1}$ factorizes through the pull-back by $\tilde\mu^{-1}$, and since the independence model
carries a product metric, the product of polytopes also carries a product metric.

Let us compare the metric $g^{(k,m)}_{K}$ from~Equation~\eqref{eq:invmetric-norm}, with the
Fisher metric $g_{(K_1,\ldots, K_k)}^{\Pbf_k}$ from Equation~\eqref{eq:indep-product} on the product of simplices $\Pbf^\circ=(\Delta_{m-1}^k)^\circ$.
In both cases, the metric is a product metric; that is, it has the form:
\begin{equation}
 g = g_{1} + \dots + g_{k},
 \label{eq:prodmetricsum}
\end{equation}
where $g_{i}$ is a metric on the $i$-th factor~$\Delta_{m-1}^\circ$.
For $g_{K}^{\Delta_{m-1}^{k}}$, $g_{i}$ is
equal to the Fisher metric on~$\Delta_{m-1}^\circ$.
However, for $g_{K}^{(k,m)}$, $g_{i}$ is equal to $1/k$ times the Fisher
metric on~$\Delta_{m-1}^\circ$.
Since this factor only depends on~$k$, it only plays a role if stochastic matrices of different sizes are compared.
The additional factor of~$1/k$ can be interpreted as the uniform distribution on~$k$ elements.
This is related to another more general class of Riemannian metrics that are used in applications;
namely, given a function $K\in\Delta_{m-1}^{k}\to \rho^K\in\R^{k}_{+}$, it is common to use product metrics with $g_{i}$ equal to
$\rho^K(i)$ times the Fisher metric on~$\Delta_{m-1}^\circ$.
When $K$ has the interpretation of a channel or when $K$ describes the policy by which a system reacts to some sensor values,
a natural possibility is to let $\rho^K$ be the stationary distribution of the channel input or of the sensor values, respectively.
We will discuss this approach in Section~\ref{sec:weighted-metrics}.

\section{Weighted Product Metrics for Conditional Models}
\label{sec:weighted-metrics}

In this section, we consider metrics on spaces of stochastic matrices defined as weighted sums of the Fisher metrics on the spaces of the matrix rows,
similar to~Equation \eqref{eq:prodmetricsum}.
This kind of metric was used initially by Amari~\cite{Amari} in order to define a natural gradient in the supervised learning context.
Later, in the context of reinforcement learning, Kakade~\cite{Kakade01} defined a natural policy gradient based on this kind of metric, which has been further developed by Peters \emph{et al}.~\cite{Peters20081180}.
Related applications within unsupervised learning have been pursued by Zahedi \emph{et al}.~\cite{zahedi}.

Consider the following weighted product Fisher metric:
\begin{equation}
 g^{\rho,m}_K = \sum_a \rho^{K}(a) g^{(m),a}_{K_{a}}, \quad\text{for all $K\in(\Delta_{m-1}^k)^\circ$}, \label{eq:Kakade}
\end{equation}
where $g^{(m),a}_{K_{a}}$ denotes the Fisher metric of $\Delta_{m-1}^\circ$ at the $a$-th row of $K$ and $\rho^K \in\Delta_{k-1}^{\circ}$ is a probability distribution over~$a$ associated with each~$K\in(\Delta_{m-1}^{k})^{\circ}$.
For example, the distribution $\rho^K$ could be the stationary distribution of sensor values observed by an agent when operating under a policy described by~$K$.

In the following, we will try to illuminate the properties of polytope embeddings that yield the metric~\eqref{eq:Kakade} as the pull-back of the Fisher information metric on a probability simplex.
We will focus on the case that $\rho^{K}=\rho$ is independent of $K$.

There are two direct ways of embedding $\Delta_{n-1}^k$ in a probability simplex.
In Section~\ref{sec:polytopes-Fisher}, we used the inverse of the moment map of an exponential family, possibly with some reference measure.
This embedding is illustrated in the left panel of Figure~\ref{fig:refmeasure}.
If we have given a fixed probability distribution $\rho\in\Delta_{k-1}^{\circ}$, there is a second natural embedding $\psi_{\rho} \colon \Delta_{m-1}^k \to \Delta_{k\cdot m -1}$ defined as follows:
\begin{equation}
 \psi_{\rho}(K)(x,y) = \rho(x)K_{x,y}\quad\text{for all $x\in[k], y\in[m]$}.
\end{equation}
If $\rho$ is the distribution of a random variable~$X$ and $K\in\Delta_{m-1}^{k}$ is the stochastic matrix describing the conditional distribution of another variable $Y$ given~$X$, then $\psi_{\rho}(K)$ is the joint distribution of $X$ and~$Y$.
Note that $\psi_{\rho}$ is an affine embedding.
See the right panel of Figure~\ref{fig:embeddings2} for an illustration.

The pull-back of the Fisher metric on $\Delta_{km-1}^\circ$ through $\psi_\rho$ is given by:
\begin{align}
g^{(km)}_{\psi_\rho(K)}({\psi_\rho}_\ast u, {\psi_\rho}_\ast v)
=&
\sum_{a,b}\sum_{c,d}\sum_{i,j} \rho(i)K_{ij} u_{ab}\frac{ \partial \log \rho(i)K_{ij}}{\partial K_{ab}}
 v_{cd} \frac{ \partial \log \rho(i)K_{ij}}{\partial K_{cd}} \nonumber \\
= & \sum_{i}\rho(i)\sum_{j}\frac{u_{ij}v_{ij} }{K_{ij}} = \sum_{i} \rho(i) g^i_{K_i} (u_i,v_i) = g^{\rho,m}_K(u,v).
\end{align}
This recovers the weighted sum of Fisher metrics from Equation~\eqref{eq:Kakade}.

\begin{figure} %[H]
\begin{center}
\quad\includegraphics{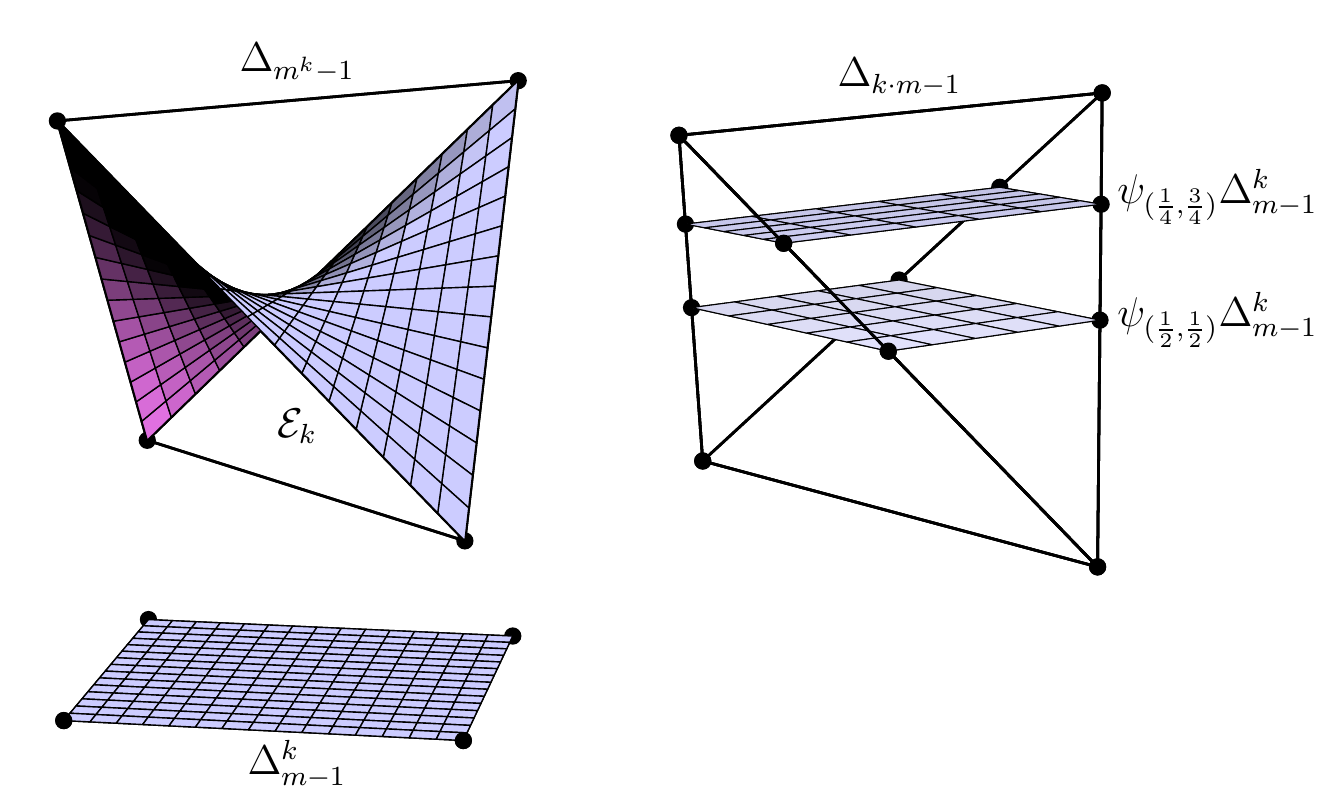}
\end{center}
\caption{
An illustration of different embeddings of the conditional polytope $\Delta_{m-1}^k$ in a probability simplex.
The left panel shows an embedding in $\Delta_{m^k-1}$ by the inverse of the moment map $\mu$ of the independence model.
The right panel shows an affine embedding in $\Delta_{k\cdot m -1}$ as a set of joint probability distributions for two different specifications of marginals.
}\label{fig:refmeasure}
\label{fig:embeddings2}
\end{figure}

%\vspace{-12pt}

Are there natural maps that leave the metrics $g^{\rho,m}$ invariant?
Let us reconsider the stochastic embeddings from Definition~\ref{def:lebanonmodi}.
Let $\ol R$ be a $k\times l$ indicator partition matrix and $R$ a stochastic partition matrix with the same block structure as $\ol R$.
Observe that to each indicator partition matrix $\ol R$ there are many compatible stochastic partition matrices $R$, but the indicator partition matrix $\ol R$ for any stochastic partition matrix $R$ is unique.
Furthermore, let $Q=\{ Q^{(a)} \}_{a\in[k]}$ be a collection of stochastic partition matrices.
The corresponding conditional embedding $\ol f$ maps $K\in \Delta_{m-1}^k$ to $\ol f(K):=\ol R^{\top}(K\otimes Q)\in\Delta_{n-1}^l$.

Let $\rho\in\Delta_{k-1}^{\circ}$. Suppose that $K$ describes the conditional distribution of $Y$ given $X$ and that
$\psi_{\rho}(K)$ describes the joint distribution of $Y$ and~$X$. As explained in Section~\ref{sec:stochsticembeddings}, the
matrix $f(P) := R^{\top}(P\otimes Q)$ describes the joint distribution of a pair of random variables $(X',Y')$, and the conditional distribution of $Y'$ given~$X'$ is given by~$\ol f(K)$. In this situation, the marginal
distribution of $X'$ is given by $\rho' = \rho R$. Therefore, the following diagram commutes:
\begin{equation}
 \label{eq:commuting-g-rho}
\begin{tikzpicture}[baseline=0pt]
 \node (C1) at (0,1) {$(\Delta_{m-1}^{k})^{\circ}$};
 \node (J1) at (2.5,1) {$\Delta_{mk-1}^{\circ}$};
 \node (C2) at (0,-1) {$(\Delta_{n-1}^{l})^{\circ}$};
 \node (J2) at (2.5,-1) {$\Delta_{nl-1}^{\circ}$};
 \begin{scope}[->,>=stealth]
 \path (C1) edge node[above] {$\psi_{\rho}$} (J1);
 \path (C2) edge node[below] {$\psi_{\rho'}$} (J2);
 \path (C1) edge node[left] {$\ol f$} (C2);
 \path (J1) edge node[right] {$f$} (J2);
 \end{scope}
\end{tikzpicture}
\end{equation}

The preceding discussion implies the first statement of the following result: %\newpage
\begin{theorem}$ $
 \label{thm:Kakade-characterization}
 \begin{itemize}
 \item
 For any $k\geq 1$ and $m\geq 2$ and any $\rho\in\Delta_{k-1}^\circ$, the Riemannian metric $g^{\rho,m}$ on
 $(\Delta_{m-1}^k)^\circ$ satisfies:
 \begin{equation}
  \label{eq:Lebanon-covariance}
  g^{\rho,m} = \ol f^{*}(g^{\rho',n}), \quad\text{ for $\rho'=\rho R$},
 \end{equation}
 for any conditional embedding $\ol f: K\mapsto \ol R(K\otimes Q)$.
 \item Conversely, suppose that for any $k\geq 1$ and $m\geq 2$ and any $\rho\in\Delta_{k-1}^\circ$, there
 is a Riemannian metric $g^{(\rho,m)}$ on $(\Delta_{m-1}^{k})^{\circ}$, such that~Equation \eqref{eq:Lebanon-covariance}
 holds for all conditional embeddings, and suppose that $g^{(\rho,m)}$ depends continuously on~$\rho$. Then, there is a
 constant $A>0$ that satisfies $g^{(\rho,m)}=A g^{\rho,m}$.
 \end{itemize}
\end{theorem}
\begin{proof}
 The first statement follows from the commutative diagram~\eqref{eq:commuting-g-rho}.
 For the second statement, denote by $\rho^{k}$ the uniform distribution on a set of $k$ elements. If $\overline{f}\colon K\mapsto\ol R(K\otimes Q)$ is a homogeneous
 conditional embedding of $\Delta_{m-1}^{k}$ in $\Delta_{n-1}^{l}$, then $R=\frac{k}{l}\ol R$ is a stochastic
 partition matrix corresponding to the partition indicator matrix~$\ol R$. Observe that $\rho^{l}=\rho^{k}R$.
 Therefore, the family of Riemannian metrics $g^{\rho^{k},m}$ on $\Delta_{m-1}^{k}$ satisfies the assumptions of
 Theorem~\ref{thm:isochar}. Therefore, there is a constant $A>0$ for which $g^{\rho^{k},m}$ equals $A /k$ times the
 product Fisher metric. This proves the statement for uniform distributions~$\rho$.

 A general distribution $\rho\in\Delta_{k-1}^\circ$ can be approximated by a distribution with rational
 probabilities. Since $g^{(\rho,m)}$ is assumed to be continuous, it suffices to prove the statement for
 rational~$\rho$. In this case, there exists a stochastic partition matrix $R$ for which $\rho':=\rho R$ is a uniform
 distribution, and so, $g^{(\rho',n)}$ is of the desired form. Equation~\eqref{eq:Lebanon-covariance} shows that
 $g^{(\rho,m)}$ is also of the desired form.
\end{proof}

\section{Gradient Fields and Replicator Equations}
\label{sec:replicator-example}

In this section, we use gradient fields in order to compare Riemannian metrics on the space $(\Delta^k_{n-1})^\circ$.

\subsection{Replicator Equations}

We start with gradient fields on the simplex $\Delta_{n-1}^\circ$.
A Riemannian metric $g$ on $\Delta_{n - 1}^\circ$ allows us to consider gradient fields of differentiable functions $F\colon \Delta_{n - 1}^\circ \to \R$.
To be more precise, consider the differential $d_p F: T_p \Delta_{n - 1}^\circ \to \R$ of $F$ in $p$.
It is a linear form on $T_p \Delta_{n - 1}^\circ$, which maps
each tangent vector $u$ to $d_p F (u) = \frac{\partial F}{\partial u} (p) \in \R$. Using the map $u \mapsto g_p(u, \cdot)$,
this linear form can be identified with a tangent vector in $T_p \Delta_{n - 1}^\circ$, which we denote by ${\rm grad}_p F $. If we choose the Fisher metric $g^{(n)}$ as the
Riemannian metric, we obtain the gradient in the following way.
First consider a differentiable extension of $F$ to the positive cone $\R_+^n$, which we will denote by the same symbol $F$.
With the partial derivatives $\partial_i F $ of $F$, the Fisher gradient of $F$ on the simplex $\Delta_{n - 1}^\circ$ is given as:
\begin{equation} \label{replicator}
 ({\rm grad}_p F)_i \; = \; p_i \left( \partial_i F (p) - \sum_{j = 1}^n p_j \, \partial_j F (p) \right), \qquad i \in [n].
\end{equation}
Note that the expression on the right-hand side of Equation~\eqref{replicator} does not depend on the particular differentiable extension of $F$ to $\R_+^n$.
The corresponding differential equation is well known in theoretical biology as the replicator equation; see~\cite{hofbauer1998evolutionary,erbay}. %(\textcolor{red}{Hofbauer \& Sigmund, Erb \& Ay zithern}).
\begin{equation}
 \dot{p}_i \; = \; p_i \left( \partial_i F (p) - \sum_{j = 1}^n p_j \, \partial_j F (p) \right), \qquad i \in [n].
\end{equation}

We now apply this gradient formula to functions that have the structure of an expectation value. Given real numbers $F_i$, $i \in [n]$,
referred to as fitness values, we consider the mean fitness:
\begin{equation}
 \bar{F}(p) := \sum_{i = 1}^n p_i \, F_i.
\end{equation}
Replacing the $p_i$ by any positive real numbers leads to a differentiable extension of $F$, also denoted by~$F$. Obviously, we have
$\partial_i F = F_i$, which leads to the following replicator equation:
\begin{equation} \label{fitness}
 \dot{p}_i \; = \; p_i \left( F_i - \bar{F}(p) \right) , \qquad i \in [n].
\end{equation}
This equation has the solution:
\begin{equation}
 p_i (t) \; = \; \frac{p_i(0) e^{t F_i}}{\sum_{j = 1}^n p_j(0) e^{t F_i}}, \qquad i \in [n].
\end{equation}
Clearly, the mean fitness will increase along this solution of the gradient field. The rate of increase can be easily calculated:
\begin{equation}
 \frac{d}{dt} \bar F\big( p(t) \big) \; = \; \sum_{i = 1}^n \dot{p}_i (t) \, F_i \; = \; \sum_{i = 1}^n p_i \left( F_i - \bar{F}(p) \right) F_i
 \; = \; \sum_{i = 1}^n p_i \left( F_i - \bar{F}(p) \right)^2 \; = \; {\rm var}_p(F) \; > \; 0.
\end{equation}
As limit points of this solution, we obtain:
\begin{equation}
 \lim_{t \to - \infty} p_i(t) \; = \;
 \left\{
  \begin{array}{c@{, \quad}l}
   \frac{p_i(0)}{\sum_{j \in {\argmin} \, F} p_j(0) } & \mbox{if $i \in {\argmin}\, F$} \\
   0      & \mbox{otherwise}
  \end{array}
 \right.
 , \qquad i \in [n],
\end{equation}
and:
\begin{equation}
 \lim_{t \to + \infty} p_i(t) \; = \;
 \left\{
  \begin{array}{c@{, \quad}l}
   \frac{p_i(0)}{\sum_{j \in {\argmax} \, F} p_j(0) } & \mbox{if $i \in {\argmax}\, F$} \\
   0      & \mbox{otherwise}
  \end{array}
 \right.
 , \qquad i \in [n].
\end{equation}

\subsection{Extension of the Replicator Equations to Stochastic Matrices}

Now, we come to the corresponding considerations of gradient fields in the context of stochastic matrices $K \in (\Delta_{n-1}^k)^\circ$.
We consider a function:
\begin{equation}
  {K} \; \mapsto \; F( {K} ) = F({K}_{11}, \dots, {K}_{1n}; {K}_{21}, \dots, {K}_{2n}; \dots ; {K}_{k1}, \dots, {K}_{k n}).
\end{equation}
One way to deal with this is to consider for each $i \in [k]$ the corresponding replicator equation:
\begin{equation}
 \dot{K}_{ij} \; = \; {K}_{ij} \left( \partial_{ij} F (K) - \sum_{j' = 1}^n {K}_{ij'} \, \partial_{ij'} F (K) \right), \qquad j \in [n].
\end{equation}
Obviously, this is the gradient field that one obtains by using the product Fisher metric on $(\Delta_{n-1}^k)^\circ$ (Equation~\eqref{eq:productFishermetric}):
\begin{equation}
 g^{(k,m)}_K (u,v) \; = \; \sum_{ij} \frac{1}{K_{ij}} \, u_{ij} v_{ij}.
\end{equation}
If we replace the metric by the weighted product Fisher metric considered by Kakade (Equation~\eqref{eq:Kakade}),
\begin{equation}
  g^{\rho,m}_K(u,v) \; = \; \sum_{ij} \frac{\rho_i}{K_{ij}} \, u_{ij} v_{ij},
\end{equation}
then we obtain
\begin{equation}
 \dot{K}_{ij} \; = \; \frac{K_{ij}}{\rho_i} \left( \partial_{ij} F (K) - \sum_{j' = 1}^n {K}_{ij'} \, \partial_{ij'} F (K) \right), \qquad j \in [n].
\end{equation}

\subsection{The Example of Mean Fitness}

Next, we want to study how the gradient flows with respect to different metrics compare.
We restrict to the class of metrics~$g^{\rho,m}$ (Equation~\eqref{eq:Kakade}), where $\rho\in\Delta_{k}^{\circ}$ is a probability distribution.
%We restrict to the class of Kakade metrics~$g^{\rho,m}$, where $\rho\in\Delta_{k}^{\circ}$ is a probability distribution.
In principle, one could drop the normalization condition $\sum_{i}\rho_{i}=1$ and allow arbitrary coefficients~$\rho_{i}$. However, it is clear
that the rate of convergence can always be increased by scaling all values $\rho_{i}$ with a common positive
factor. Therefore, some normalization condition is needed for~$\rho$.

With a probability distribution $p \in \Delta_{k-1}^\circ$ and fitness values $F_{ij}$, let us consider again the example of an expectation value function:
\begin{equation}
 \bar F(K) \; = \; \sum_{i = 1}^k p_i \sum_{j = 1}^n {K}_{ij} \, F_{ij}.
\end{equation}
With $\partial_{ij} \bar F (\pi) = p_i \, F_{ij}$, this leads to:
\begin{equation}
 \dot{K}_{ij} \; = \; \frac{p_{i}}{\rho_{i}} \, {K}_{ij} \left( F_{ij} - \sum_{j' = 1}^n {K}_{ij'} \, F_{ij'} \right), \qquad j \in [n].
\end{equation}
The corresponding solutions are given by:
\begin{equation}
 {K}_{ij} (t) \; = \; \frac{K_{ij}(0) \, e^{t \, \frac{p_{i}}{\rho_{i}} F_{ij}}}{\sum_{j' = 1}^n {K}_{ij'}(0) \, e^{t \, \frac{p_{i}}{\rho_{i}} F_{ij'}}}, \qquad i \in [n].
\end{equation}
Since $\argmax(\frac{p_{i}}{\rho_{i}} \, F_{i\cdot})$ and $\argmin(\frac{p_{i}}{\rho_{i}} \, F_{i\cdot})$ are independent of $\rho_i>0$, the limit points are given independently of the chosen $\rho$ as:
\begin{equation}
 \lim_{t \to - \infty} {K}_{ij}(t) \; = \;
 \left\{
  \begin{array}{c@{, \quad}l}
   \frac{K_{ij}(0)}{\sum_{j' \in {\argmin} \, F_{i\cdot}} {K}_{ij'}(0) } & \mbox{if $j \in {\argmin}\, F_{i\cdot}$} \\
   0      & \mbox{otherwise}
  \end{array}
 \right.
 , \qquad i \in [n],
\end{equation}
and:
\begin{equation}
 \lim_{t \to + \infty} {K}_{ij}(t) \; = \;
 \left\{
  \begin{array}{c@{, \quad}l}
   \frac{K_{ij}(0)}{\sum_{j' \in {\argmax} \, F_{i\cdot}} {K}_{ij'}(0) } & \mbox{if $j \in {\argmax}\, F_{i\cdot}$} \\
   0      & \mbox{otherwise}
  \end{array}
 \right.
 , \qquad i \in [n].
\end{equation}
This is consistent with the fact that the critical points of gradient fields are independent of the chosen Riemannian metric.
However, the speed of convergence does depend on the metric:

For each $i$, let $G_{i}=\max_{j}F_{ij}$ and $g_{i}=\max_{j\notin\argmax(F_{ij})}F_{ij}$ be the largest and second-largest values in the $i$-th row of $F_{ij}$, respectively. Then, as:
$t\to\infty$,
\begin{equation}
 K_{ij}(t) \; =\;
 \left\{
 \begin{array}{c@{, \quad}l}
 1 - O(\exp(-\frac{p_{i}}{\rho_{i}}(G_{i}-g_{i})t) & \mbox{ if $i\in{\argmax}\, F_{i\cdot}$} \\
 O(\exp(-\frac{p_{i}}{\rho_{i}}(G_{i}-F_{ij})t) & \mbox{otherwise}
 \end{array}
 \right.
\end{equation}
Therefore,
\begin{multline}
 \bar F(K(t))
 = \sum_{i}p_{i} \sum_{j \in \argmax F_{i\cdot}} F_{ij} + O\left(\exp(-\frac{p_{i}}{\rho_{i}}(G_{i}-g_{i})t)\right) \\
 = \sum_{i}p_{i}\sum_{j \in \argmax F_{i\cdot}} F_{ij} + O\left(\exp(-\inf_{i}\left\{\frac{p_{i}}{\rho_{i}}(G_{i}-g_{i})\right\}t)\right).
\end{multline}
Thus, in the long run, the rate of convergence is given by~$\inf_{i}\{\frac{p_{i}}{\rho_{i}}(G_{i}-g_{i})\}$, which
depends on the parameter~$\rho$ of the metric.
As a result, in this case study, the optimal choice of $\rho_{i}$, \emph{i.e.}, with the largest convergence rate, can be
computed if the numbers $G_{i}$ and~$g_{i}$ are known.

Consider, for example, the case that the differences $G_{i}-g_{i}$ are of comparable sizes for all~$i$. Then, we need to
find the choice of $\rho$ that maximizes $\inf_{i}\{\frac{p_{i}}{\rho_{i}}\}$.
Clearly, $\inf_{i}\{\frac{p_{i}}{\rho_{i}}\}\le 1$ (since there is always an index $i$ with $p_{i}\le\rho_{i}$). Equality is
attained for the choice $\rho_{i}=p_{i}$. Thus, we recover the choice of Kakade.

\section{Conclusions}
\label{sec:conclusions}

So, which Riemannian metric should one use in practice on the set of stochastic matrices, $(\Delta_{m-1}^{k})^{\circ}$?
The results provided in this manuscript give different answers, depending on the approach. In all cases, the characterized Riemannian metrics are products of Fisher metrics with suitable factor weights.
Theorem~\ref{thm:isochar} suggests to use a factor weight proportional to~$1/k$, and Theorem~\ref{thm:metric-on-polytopes} suggests to use a constant weight independent of $k$.
In many cases, it is possible to work within a single conditional polytope~$(\Delta_{m-1}^{k})^{\circ}$ and a fixed $k$, and then, these two results are basically equivalent.
On the other hand, Theorem~\ref{thm:Kakade-characterization} gives an answer that allows arbitrary factor weights~$\rho$.

Which metric performs best obviously depends on the concrete application.
The first observation is that in order to use the metric $g^{\rho,m}$ of Theorem~\ref{thm:Kakade-characterization}, it is necessary to know~$\rho$.
If the problem at hand suggests a natural marginal distribution~$\rho$, then it is natural to make use of it and choose the metric~$g^{\rho,m}$.
Even if $\rho$ is not known at the beginning, a learning system might try to learn it to improve its performance.

On the other hand, there may be situations where there is no natural choice of the weights~$\rho$. Observe that $\rho$ breaks the symmetry of permuting the rows of a stochastic matrix.
This is also expressed by the structural difference between Theorems~\ref{thm:isochar} and~\ref{thm:metric-on-polytopes} on the one side and Theorem~\ref{thm:Kakade-characterization} on the other.
While the first two theorems provide an invariance metric characterization, Theorem~\ref{thm:Kakade-characterization} provides a
``covariance'' classification; that is, the metrics $g^{\rho,m}$ are not invariant under conditional embeddings,
but they transform in a controlled manner.
This again illustrates that the choice of a metric should depend on which mappings are natural to consider, e.g., which mappings describe the symmetries of a given problem.

For example, consider a utility function of the form $F=\sum_{i}\rho_i \sum_j K_{ij} F_{ij}$.
Row permutations do not leave $g^{\rho,m}$ invariant (for a general $\rho$),
but they are not symmetries of the utility function $F$, either, and hence, they are not very natural mappings to consider.
However, row permutations transform the metric $g^{\rho,m}$
and the utility function in a controlled manner; in such a way that the two transformations match.
Therefore, in this case, it is natural to use~$g^{\rho,m}$.
On the other hand, when studying problems that are symmetric under all row
permutations, it is more natural to use the invariant metric~$g^{(k,m)}$.

%\vspace{12pt}
\appendix %{\noindent\textbf{Appendix}}
\section*{Appendix}
\setcounter{equation}{0}
\renewcommand\theequation{A\arabic{equation}}
\renewcommand\thesubsection{\Alph{subsection}}

\subsection{Conditions for Positive Definiteness}
\label{section:positive}

Equation~\eqref{eq:Lebanon-metrics} in Lebanon's Theorem~\ref{thm:Lebanon} defines a Riemannian metrics whenever it defines a positive-definite quadratic form.
The next proposition gives sufficient and necessary conditions for which this is the case.

\begin{proposition}
% \label{proposition:positivedef}
For each pair $k\geq 1$ and $m\geq 2$,
consider the tensor on $\R_+^{k\times m}$ defined by:
\begin{equation}
  g_M^{(k,m)}(\partial_{ab}, \partial_{cd} )
  = A(|M|) +\delta_{ac} \left( \frac{B(|M|)}{|M_a|} + \delta_{bd} \frac{C(|M|)}{M_{ab}} \right)
 \end{equation}
 for some differentiable functions $A, B, C \in C^\infty(\R_+)$.
The tensor $g^{(k,m)}$ defines a Riemannian metric for all $k$ and $m$ if and only if $C(\alpha)>0$, $B(\alpha)+C(\alpha)>0$ and $A(\alpha)+B(\alpha)+C(\alpha)>0$ for all $\alpha\in\R_+$.
\end{proposition}
\begin{proof}
The tensors are Riemannian metrics when:
\begin{equation}
g_M^{(k,m)}(V) =
A(|M|) (\sum_{ab} V_{ab})^2 + B(|M|)\sum_a\frac{|M|}{|M_a|}(\sum_b V_{ab})^2 + C(|M|)\sum_{ab}\frac{|M|}{M_{ab}} V_{ab}^2
\label{equation:quadraticform}
\end{equation}
is strictly positive for all non-zero $V\in\R^{k\times m}$, for all $M\in\R^{k\times m}_+$.

We can derive necessary conditions on the functions $A,B,C$ from some basic observations.
% For instance, evaluating $g^{(k,m)}_M(\partial_{ab})$
Choosing $V=\partial_{ab}$ in~Equation \eqref{equation:quadraticform} shows that $A(|M|) +\frac{|M|}{|M_a|}B(|M|) +\frac{|M|}{M_{ab}}C(|M|)$ has to be positive for all $a\in[k],b\in[m]$, for all $M\in\R^{k\times m}_+$.
Since $M_{ab}$ can be arbitrarily small for fixed $|M|$ and $|M_a|$, we see that $C$ has to be non-negative.
Since we can choose $|M_a|\approx M_{ab}\ll |M|$ for a fixed $|M|$, we find that $B+C$ has to be non-negative.
Further, since we can choose $M_{ab}\approx |M_a| \approx |M|$ for a given $|M|$, we find that $A+B+C$ has to be non-negative.
This shows that the quadratic form is positive definite only if $C\geq 0$, $B+C\geq 0$, $A+B+C\geq 0$.
Since the cone of positive definite matrices is open, these inequalities have to be strictly satisfied.
In the following, we study sufficient conditions.

For any given $M\in\R^{k\times m}_+$, we can write Equation~\eqref{equation:quadraticform} as a product $V^\top G V$, for all $V\in\R^{km}$,
where $G = G_A + G_B + G_C \in\R^{km\times km}$ is the sum of a matrix $G_A$ with all entries equal to $A(|M|)$, a block diagonal matrix $G_B$ whose $a$-th block has all entries equal to $\frac{|M|}{|M_a|} B(|M|)$,
and a diagonal matrix $G_C$ with diagonal entries equal to $\frac{|M|}{M_{ab}} C(|M|)$.
The matrix $G$ is obviously symmetric, and by Sylvester's criterion, it is positive definite iff all its leading principal minors are positive.
We can evaluate the minors using Sylvester's determinant theorem.
That theorem states that for any invertible $m\times m$ matrix $X$, an $m\times n$ matrix $Y$ and an $n\times m$ matrix $Z$,
one has the equality $\det(X + YZ)=\det(X)\det(I_n + Z X^{-1}Y)$.

Let us consider a leading square block $G'$, consisting of all entries $G_{ab,cd}$ of $G$ with row-index pairs $(a,b)$ satisfying $b\in [m]$ for all $a<a'$ and $b\leq b'$ for $a=a'$ for some $a'\leq k$ and $b'\leq m$; and the same restriction for the column index pairs.
The corresponding block $G'_A+G'_B$ can be written as the rank-$a'$ matrix $YZ$,
with $Y$ consisting of columns $\mathds{1}_a$ for all $a\leq a'$
and $Z$ consisting of rows $A + \mathds{1}_a \tfrac{|M|}{|M_a|}B $ for all $a\leq a'$.
Hence, the determinant of $G'$ is equal to:
\begin{equation}
\det(G') = \det(G'_C)\cdot\det(I_{a'} + Z {G'}_C^{-1} Y ).
\label{eq:minor}
\end{equation}
Since ${G'}_C$ is diagonal, the first term is just:
\begin{equation}
\det(G'_C) = \left( \prod_{a<a'} \prod_b \tfrac{|M|}{M_{ab}} C \right) \left( \prod_{b\leq b'} \tfrac{|M|}{M_{a' b}} C \right).
\end{equation}
The matrix in the second term of Equation~\eqref{eq:minor} is given by:
\begin{multline}
I_{a'} + Z {G'}_C^{-1} Y = \\
\frac{1}{C}
\begin{pmatrix}
C+ B\!\!& 		& 	 & \\
	& \!\ddots& 	 & \\
	&		&\!\!C+ B& \\
	&		&	 & C+ \tfrac{\sum_{b\leq b'} M_{a' b}}{|M_{a'}|} B
\end{pmatrix}
+
\frac{1}{C}
\begin{pmatrix}
 \tfrac{|M_1|}{|M|}A & \cdots 	& \tfrac{|M_{a'-1}|}{|M|}A & \tfrac{\sum_{b\leq b'} M_{a' b}}{|M|} A \\
 \vdots & 		& \vdots	 & \vdots \\
 \tfrac{|M_1|}{|M|}A & \cdots	& \tfrac{|M_{a'-1}|}{|M|}A & \tfrac{\sum_{b\leq b'} M_{a' b}}{|M|} A
\end{pmatrix}.
\end{multline}
By Sylvester's determinant theorem, we have:
\begin{align}
\det(I_{a'} + Z {G'}_C^{-1} Y )
 &= C^{-a'} (C+B)^{a'-1} (C+ \tfrac{\sum_{b\leq b'} M_{a' b}}{|M_{a'}|} B) (1 + \sum_{a<a'} \frac{\tfrac{|M_a|}{|M|} A}{C+B} + \frac{\tfrac{\sum_{b\leq b'} M_{a' b}}{|M|} A}{C+ \tfrac{\sum_{b\leq b'} M_{a' b}}{|M_{a'}|}B} ) \nonumber\\
 &=
\left(\prod_{ a} \frac{C+ B_{a}}{C}\right) \left(1 + \sum_{a} \frac{ A_a}{C +B_a} \right),
\end{align}
where $A_a = \frac{|M_a|}{|M|}A$ for $a<a'$ and $A_{a'} = \frac{\sum_{b\leq b'}M_{a' b}}{|M|}A$, and $B_a= B$ for $a<a'$ and $B_{a'}=\frac{\sum_{b\leq b'} M_{a' b}}{|M_{a'}|}B$.

This shows that the matrix $G$ is positive definite for all $M$ if and only if $C>0$, $C+B>0$ and $\left(1 + \sum_{a\leq a'} \frac{ A_a}{C +B_a} \right)>0$ for all $a'$ and $b'$.
The latter inequality is satisfied whenever $A+B+C>0$.
This completes the proof.
\end{proof}

\subsection{Proofs of the Invariance Characterization}
\label{sec:proofs-invariance}

The following lemma follows directly from the definition and contains all the technical details we need for the proofs.
\begin{lemma}
The push-forward $f_\ast\colon T_M\R^{k\times m}_+ \to T_{f(M)}\R^{l\times n}_+ $ of a map $f\in\hat\Fcal_{k,m}^{l,n}$ is given by:
\begin{equation}
 f_\ast(\partial_{ab}) = \sum_{i=1}^l \sum_{j=1}^n \ol R_{ai} Q^{(a)}_{bj} \partial'_{ij},
\end{equation}
and the pull-back of a metric $g^{(l,n)}$ on $\R^{l\times n}_+$ through $f$ is given by:
\begin{equation}
(f^\ast g^{(l,n)})_M (\partial_{ab},\partial_{cd})
=
g^{(l,n)}_{f(M)} (f_\ast \partial_{ab}, f_\ast \partial_{cd}) % \nonumber \\
=
\sum_{i=1}^l \sum_{j=1}^n \sum_{s=1}^l \sum_{t=1}^n \ol R_{ai} \ol R_{cs} Q^{(a)}_{bj} Q^{(c)}_{dt} g^{(l,n)}_{f(M)} (\partial'_{ij}, \partial'_{st}).
\label{eq:pullbackmetric}
\end{equation}
\end{lemma}

\begin{proof}[Proof of Theorem~\ref{thm:isochar}]
We follow the strategy of~\cite{Campbell,Lebanon05}.
The idea is to consider subclasses of maps from the class $\Fcal_{k,m}^{l,n}$
and to evaluate their push-forward and pull-back maps together with the isometry requirement.
This yields restrictions on the possible metrics,
eventually fully characterizing them.

\begin{proof}[First]
Consider the maps $h_{\pi,\sigma} \in \mathcal{F}_{k,m}^{k,m}$, resulting from permutation matrices $Q^{(a)} = P_{\pi^{a}}$, $\pi^a\colon [m]\to[m]$ for all $a\in[k]$,
and $\overline{R}=P_{\sigma}$, $\sigma\colon [k]\to[k]$.
Requiring isometry yields:
\begin{eqnarray}
({h_{\pi,\sigma}})_\ast(\partial_{ab}) & = & \partial'_{\sigma(a)\,\pi^{a}(b)} \\
g^{(k,m)}_M(\partial_{ab},\partial_{cd}) & = &
g^{(k,m)}_{h_{\pi,\sigma}(M)} (\partial_{\sigma(a)\,\pi^{(a)}(b)}, \partial_{ \sigma(c)\,\pi^{(c)}(d) } ).
\end{eqnarray}
\phantom\qedhere
\end{proof}

\begin{proof}[Second]
Consider the maps $r_{zw}\in\mathcal{F}_{k,m}^{kz, mw}$ defined by $Q^{(1)}=\cdots = Q^{(k)}\in\R^{m\times m w}$ and $\overline{R}\in\R^{k\times k z}$ being uniform.
In this case, for some permutations $\pi$ and $\sigma$,
\begin{eqnarray}
({r_{z w} })_\ast (\partial_{ab}) & = & \frac{1}{w} \sum_{i=1}^z \sum_{j=1}^w \partial'_{\sigma^{(a)}(i) \, \pi^{(b)}(j)} \\
({ r_{z w} }^\ast g^{(kz,mw)})_M(\partial_{ab},\partial_{cd}) & = &
\frac{1}{w^2} \sum_{i=1}^z \sum_{j=1}^w \sum_{s=1}^z \sum_{t=1}^w g^{(kz,mw)}_{r_{zw}(M)}(\partial'_{\sigma^{(a)}(i)\, \pi^{(b)}(j)} , \partial'_{\sigma^{(c)}(s)\, \pi^{(d)}(t)}).
\end{eqnarray}
\phantom\qedhere
\end{proof}

\begin{proof}[Third]
For a rational matrix $M=\frac{1}{Z}\tilde M$ with $\tilde M\in\mathbb{N}^{k\times m}$ and row-sum $|\tilde M_a| = N\in\mathbb{N}$ for all $a\in[k]$,
consider the map $v_M\in\mathcal{F}_{k,m}^{z k, N}$ that maps $M$ to a constant matrix.
In this case, $\overline{R}\in\R^{k\times k z}$ and $Q^{(a)}$ has the $b$-th row
with $|\tilde M_{ab}|$ entries with value $\frac{1}{|\tilde M_{ab}|}$ at positions
$\pi^{(ab)}([\tilde M_{ab}])\subseteq[N]$, and:
\begin{eqnarray}
({v_M})_\ast(\partial_{ab}) & = &
\frac{1}{\tilde M_{ab}} \sum_{i=1}^k\sum_{j=1}^{\tilde M_{ab}} \partial'_{\sigma^{(a)}(i)\, \pi^{(ab)}(j)} \\
({v_M}^\ast g^{(k z,N)} )_{M} (\partial_{ab} , \partial_{cd} ) & = &
\frac{1}{\tilde M_{ab}} \frac{1}{\tilde M_{cd}} \sum_{i=1}^z\sum_{j=1}^{\tilde M_{ab}} \sum_{s=1}^z\sum_{t=1}^{\tilde M_{cd}}
g^{(kz,N)}_{v_M(M)} (\partial'_{\sigma^{(a)}(i)\, \pi^{(ab)}(j)} , \partial'_{\sigma^{(c)}(s)\, \pi^{(cd)}(t) }).
\end{eqnarray}
\phantom\qedhere
\end{proof}

\noindent\textit{Step 1:} % $g^{(k,m)}_M(\partial_{ab}, \partial_{cd})$ for
$a\neq c$.
Consider a constant matrix $M=U$. Then:
\begin{equation}
g^{(k,m)}_U(\partial_{a_1 b_1} , \partial_{c_1 d_1} )
 = g_{h_{\pi, \sigma}(U)}^{(k,m)} (\partial_{a_2 b_2}, \partial_{c_2 d_2})
 = g_U^{(k,m)} (\partial_{a_2 b_2}, \partial_{c_2 d_2}).
\end{equation}
This implies that $g^{(k,m)}_U(\partial_{ab} , \partial_{cd} ) = \hat A(k,m)$ when $a\neq c$.

Using the second type of map, we get:
\begin{align}
\hat A(k,m)
& = \frac{z^2 w^2}{w^2} \hat A (kz, mw),
\end{align}
which implies $g^{(k,m)}_U(\partial_{a b} , \partial_{c d} ) = \frac{A}{k^2}$, when $a\neq c$.
Considering a rational matrix $M$ and the map $v_M$ yields:
\begin{equation}
g_M^{(k,m)}(\partial_{ab} , \partial_{c,d} ) = \frac{A}{k^2}.
\end{equation}

\noindent\textit{Step 2:}
% $g_M^{(k,m)}(\partial_{ab}, \partial_{ad})$,
$b\neq d$.
By similar arguments as in Part~1, $g_U^{(k,m)}(\partial_{ab},\partial_{ad}) = \hat B(k,m)$.
Evaluating the map $r_{zw}$ yields:
\begin{align}
\hat B(k,m)
& = \frac{z w^2}{w^2} \hat B(kz,mw) + \frac{z(z-1)w^2}{w^2}\frac{A}{(kz)^2} \nonumber \\
& = z \hat B(kz,mw) + \frac{ z-1}{ z} \frac{A}{k^2 },
\intertext{and therefore,}
\frac{1}{z} \left( \hat B(k,m) - \frac{A}{ k^2} \right) & = \hat B(kz,mw) - \frac{A}{ (k z)^2},
\end{align}
which implies that $\left( \hat B(k,m) - \frac{A}{ k^2} \right)$ is independent of $m$ and scales with the inverse of $k$, such that it can be written as $\frac{B}{k}$.
Rearranging the terms yields $g_U^{(k,m)}(\partial_{ab}, \partial_{ad}) = \frac{A}{k^2} + \frac{B}{k}$, for $b\neq d$.

For a rational matrix $M$, the pull-back through $v_M$ shows then:
\begin{equation}
g_M^{(k,m)}(\partial_{ab},\partial_{cd})
% &
= z\frac{\tilde M_{ab}\tilde M_{ad}}{\tilde M_{ab}\tilde M_{ad}} \left( \frac{A}{(k z)^2} +\frac{B}{kz} \right) + \frac{z(z-1)\tilde M_{ab} \tilde M_{ad}}{\tilde M_{ab}\tilde M_{ad}}\frac{A}{(kz)^2}
% \nonumber \\&
 = \frac{A}{k^2} + \frac{B}{k}.
\end{equation}

\noindent\textit{Step 3:}
% $g_M^{(k,m)}(\partial_{ab},\partial_{cd})$ for
$a=c$ and $b=d$.
In this case, $g_U^{(k,m)}(\partial_{a_1 b_1},\partial_{a_1 b_1}) = g_U^{(k,m)}(\partial_{a_2 b_2},\partial_{a_2 b_2}) = \hat C(k,m)$, and:
\begin{align}
\hat C(k,m)
& = \frac{1}{w^2} zw \hat C(kz, mw) + \frac{1}{w^2} z w (w-1) \left( \frac{A}{(kz)^2} +\frac{B}{kz} \right) + \frac{1}{w^2} z w^2 z(z-1) \frac{A}{(kz)^2} \nonumber \\
& = \frac{z}{w} \hat C(kz, mw) + (1-\frac{1}{zw})\frac{A}{k^2} + (1-\frac{z}{zw})\frac{B}{k},
\end{align}
which implies:
\begin{equation}
\frac{k}{m}\left( \hat C(k,m) -\frac{A}{k^2} - \frac{B}{k} \right) =\frac{kz}{mw} \left( \tilde C(kz, mw) -\frac{A}{(kz)^2} -\frac{B}{kz} \right),
\end{equation}
such that the left-hand side is a constant $C$,
and $g_U^{(k,m)}(\partial_{ab},\partial_{ab}) = \frac{A}{k^2} +\frac{B}{k} +\frac{m}{k}C$.
Now, for a rational matrix $M$, pulling back through $v_M$ gives:
\begin{align}
g_M^{(k,m)}(\partial_{ab}, \partial_{ab})
& = \frac{1}{\tilde M_{ab}^2} \tilde M_{ab} \left( \frac{A}{k^2} +\frac{B}{k} + \frac{|\tilde M_a|}{k}C \right)
+ \frac{1}{\tilde M_{ab}^2} \tilde M_{ab}(\tilde M_{ab} -1) \left( \frac{A}{k^2} +\frac{B}{k} \right)
+ 0 \nonumber \\
& = \frac{A}{k^2} + \frac{B}{k} + \frac{|\tilde M_a|}{\tilde M_{ab} k} C \nonumber \\
& = \frac{A}{k^2} + k\frac{B}{k^2} + \frac{|M|}{ M_{ab}} \frac{C}{k^2}.
\end{align}

Summarizing, we found:
\begin{equation}
g_{M}^{(k,m)} (\partial_{ab},\partial_{cd}) =
\frac{A}{k^2} + \delta_{ac}\left( k\frac{B}{k^2} + \delta_{bd} \frac{|M|}{ M_{ab}} \frac{C}{k^2} \right),
\end{equation}
which proves the first statement. The second statement follows by plugging Equation~\eqref{eq:invmetric} into Equation~\eqref{eq:pullbackmetric}.
Finally, the statement about the positive-definiteness is a direct consequence of Proposition~\ref{proposition:positivedef}.
\end{proof}

\begin{proof}[Proof of Theorem~\ref{thm:isocharstrong}]
 Suppose, contrary to the claim, that a family of metrics $g_{M}^{(k,m)}$ exists, which is invariant with respect to any conditional embedding.
 By Theorem~\ref{thm:isochar}, these metrics are of the form of~Equation \eqref{eq:invmetric}.
 To prove the claim, we only need to show that $A$, $B$ and $C$ vanish.
 In the following, we study conditional embeddings where $Q$ consists of identity matrices and evaluate the isometry requirement $(f^\ast g^{(l,n)})_M(\partial_{ab},\partial_{cd}) = g^{(k,m)}_M(\partial_{ab},\partial_{cd})$.

 \noindent\textit{Step 1:} 
 In the case $a\neq c$, we obtain from the invariance requirement and Equation~\eqref{eq:pullbackmetric}, that:
 \begin{equation}
 \label{eq:Azero}
 \frac{A}{k^2} =
 |\overline{R}_a | |\overline{R}_c | \frac{A}{l^2}.
 \end{equation}
 Observe that:
 \begin{equation}
 \frac{1}{k}\sum_{i=1}^{k} |\overline{R}_{i}| =\frac1k |\overline{R}| = \frac lk.
 \end{equation}
 In fact, $|\overline{R}_{i}|$ is the cardinality of the $i$-th block of the partition belonging to~$\overline{R}$.
 Therefore, if we choose $\overline{R}$ to be the partition indicator matrix of a partition that is not homogeneous and in which
 $|\overline{R}_{a}| > l/k$ and~$|\overline{R}_{c}| > l/k$, then~Equation \eqref{eq:Azero} implies that~$A=0$.

 \noindent\textit{Step 2:} 
 In the case $a=c$ and $b\neq d$, we obtain from invariance and~Equation \eqref{eq:pullbackmetric}, that:
 \begin{equation}
% \label{eq:Bzero}
 \frac{B}{k} = \sum_{i=1}^{l}\sum_{s=1}^{l} \overline{R}_{ai}\overline{R}_{as} \delta_{is}\frac{B}{l}
 = |\overline{R}_{a}|\frac{B}{l}.
 \end{equation}
 Again, we may chose $\overline{R}_a$ in such a way that $|\overline{R}_{a}| \neq \frac kl$ and find that~$B=0$.

 \noindent\textit{Step 3:} 
 Finally, in the case $a=c$ and $b=d$, we obtain from invariance and~Equation \eqref{eq:pullbackmetric}, that:
 \begin{equation}
 \frac{C |M|}{k^{2} M_{ab}} = \sum_{i=1}^{l}\sum_{s=1}^{l} \overline{R}_{ai}\overline{R}_{as} \delta_{i,s}\frac{C |\overline{R}^{\top}M|}{l^{2}(\overline{R}^{\top}M)_{ib}}
 = % \sum_{i=1}^{l} R_{a,i}
 |\overline{R}_{a}|\frac{C |\overline{R}^{\top}M|}{l^{2}M_{ab}}.
 \end{equation}
If we chose $\overline{R}_a$, such that $|\overline{R}_{a}|\neq\frac{|M|}{|\overline{R}^\top M|}$, then
 we see that~$C=0$.
 Therefore, $g^{(k,m)}$ is the zero-tensor, which is not a metric.
\end{proof}

\subsubsection*{Acknowledgments}
The authors are grateful to Keyan Zahedi for discussions related to policy gradient methods in robotics applications.
Guido Mont\'ufar thanks the Santa Fe Institute for hosting him during the initial work on this article.
Johannes Rauh acknowledges support by the VW
 Foundation.
This work was supported in part by the DFG Priority Program, Autonomous Learning (DFG-SPP 1527).

%\section*{\noindent Author Contributions}
%\vspace{12pt}
%
%All authors contributed to the design of the research. The research was carried out by all authors, with main contributions by Guido Mont\'ufar and Johannes Rauh. The manuscript was written by Guido Mont\'ufar, Johannes Rauh and Nihat Ay. 
%All authors read and approved the final manuscript. 

%\conflictofinterests{Conflicts of Interest}
%The authors declare no conflict of interests.

\renewcommand{\baselinestretch}{.95}
\bibliographystyle{abbrv}
\bibliography{../referenzen}

\end{document}